\documentclass[11pt]{amsart}

\newtheorem{theorem}{Theorem}[section]

\newtheorem{corollary}[theorem]{Corollary}

\newtheorem{lemma}[theorem]{Lemma}

\newtheorem{proposition}[theorem]{Proposition}

\theoremstyle{definition}
\newtheorem{remark}[theorem]{Remark}

\usepackage[colorlinks, bookmarks=true]{hyperref}
\usepackage{color,graphicx,shortvrb}

\usepackage{enumerate}
\usepackage{amsmath}
\usepackage{amssymb}

\def\RR{{\mathbb{R}}}
\def\NN{{\mathbb{N}}}
\def\11{\textbf{$1$}}
\def\CC{{\mathbb{C}}}

\def\FF{{\mathbb{F}}}

\newcommand{\supp}{\mathrm{supp}}

\newcommand{\dual}{{\mathfrak{D}}}
\newcommand{\ovdual}{{\mathfrak{D}}^o}

\newcommand{\fatou}{{\mathfrak{f}}}
\newcommand{\cq}{{\mathfrak{C}}_q}

\newcommand{\proj}{\mathbf{P}}
\newcommand{\sign}{\mathrm{sign} \, }

\newcommand{\one}{\mathbf{1}}

\newcommand{\ran}{\mathrm{ran} \,}

\newcommand{\spn}{\mathrm{span}}
\newcommand{\vr}{\varepsilon}

\newcommand{\DP}{\mathrm{DP}}

\newcommand{\ball}{\mathbf{B}}
\newcommand{\sph}{\mathbf{S}}
\newcommand{\expe}{\mathbb{E}}
\newcommand{\prob}{\mathbb{P}}

\def \eqalign#1{\null\,\vcenter{\openup\jot 
   \ialign{\strut\hfil$\displaystyle{##}$&$
      \displaystyle{{}##}$\hfil \crcr#1\crcr}}\,}

\begin{document}

\numberwithin{equation}{section}

\title[Almost disjointness preservers]{Almost disjointness preservers}

 \author[T. Oikhberg]{Timur Oikhberg}
 \address{Dept.~of Mathematics, University of Illinois, Urbana IL 61801, USA}
 \email{oikhberg@illinois.edu}

 \author[P. Tradacete]{Pedro Tradacete}
\address{Mathematics Department, Universidad Carlos III de Madrid, E-28911 Legan\'es, Madrid, Spain.}
\email{ptradace@math.uc3m.es}

\thanks{P.T. partially supported by the Spanish Government grants MTM2013-40985, MTM2012-31286, and Grupo UCM 910346. }



\date{}
\maketitle

\begin{abstract}
We study the stability of disjointness preservers on Banach lattices.
In many cases, we prove that an ``almost disjointness preserving''
operator is well approximable by a disjointess preserving one.
However, this approximation is not always possible, as our
examples show.
\end{abstract}

\tableofcontents

\maketitle
\thispagestyle{empty}

\section{Introduction}\label{s:intro}

Recall that an operator $T$ between Banach lattices $E$ and $F$ is called
\emph{disjointness preserving} (\emph{DP} for short) if $Tx \perp Ty$
whenever $x \perp y$. Such operators have been
investigated intensively, and are known to possess many remarkable
properties (see e.g. \cite{Are},  \cite[Chapter 3]{M-N}, or the survey paper \cite{Hui}).
For instance, it is known that any DP operator on $C(K)$ is a
weighted composition \cite[Section 3.1]{M-N}. In \cite{OPP},
a similar result was shown for DP maps on K\"othe spaces.
For many other kinds of spaces, the general form of a DP map
is also known (see e.g. \cite{Ar04}, \cite{KN01}, \cite{Leu}).
Compact DP maps on $C(K)$ have been described in \cite{LW}.
Moreover, the inverse of a DP map is again DP, see \cite{Are}.

In this paper, we investigate the ``stability'' of being disjointness preserving.
To be more specific, suppose $E$ and $F$ are Banach lattices.
We say that an operator $T : E \to F$
is \emph{$\vr$-disjointness preserving} (\emph{$\vr$-DP} for short) if,
for any disjoint $x, y \in E$,
$$
\| |Tx| \wedge |Ty| \| \leq \vr \max\{ \|x\|, \|y\| \}.
$$
Note that $0$-DP operators are precisely the disjointness preserving operators.

Note that if $T$ is $\vr$-DP, then for any scalar $\lambda$,
$\lambda T$ is $|\lambda| \vr$-DP. Clearly, every operator $T$ is $\|T\|$-DP, so the above notion is only interesting for $\vr<\|T\|$.

The goal of this paper is to investigate the properties of $\vr$-DP operators, and furthermore, to determine whether such operators
can be approximated by disjointness preserving ones. More precisely:
for what $\vr$-DP operators $T$ does there exist a DP map $S$ with
$\|T-S\| \leq \phi(\vr,\|T\|)$, where $\lim_{\vr \to 0} \phi(\vr,t) = 0$
for every $t$?

This question has been considered previously on spaces of continuous functions. Namely, G. Dolinar \cite{Do02} (and later J. Araujo and J. Font
\cite{AF_JLMS, AF_PAMS, AF_JAT}, as well as R. Kantrowitz and M. Neumann \cite{KN})
considered a formally different notion of almost disjointness preserving operators between $C(K)$ spaces.
More precisely, suppose $E = C(K_E)$ and $F = C(K_F)$. We say that $T : E \to F$
is \emph{Dolinar $\vr-\DP$} if
$$
\| (Tx) (Ty) \| \leq \vr \|x\| \|y\|
$$
for any disjoint $x$ and $y$. It is easy to see that if $T : C(K_E) \to C(K_F)$ is Dolinar $\vr$-DP then it is $\sqrt{\vr}-\DP$; and, in the converse direction, if $T : C(K_E) \to C(K_F)$ is $\vr$-DP, then it is Dolinar $\|T\| \vr$-DP. Improving the results of \cite{Do02}, in \cite{AF_JLMS} the authors showed that if $T$ is a Dolinar $\vr-\DP$ contraction ($0 < \vr < 2/17$), then there exists a (disjointness
preserving) weighted composition operator $S$ so that $\|T - S\| < \sqrt{17\vr/2}$.
\cite{AF_PAMS} improves on this for linear functionals.

The paper is organized as follows: Section \ref{s:basics} is devoted to collecting basic facts
about $\vr$-DP operators. In Section \ref{s:ineq}, we establish
a probablistic inequality (to be used throughout our work), and list some of its consequences.

In Section \ref{s:c_0} we show that positive $\vr$-DP operators from $c_0$ or $c$
into a Banach lattice with the Fatou Property
can be nicely approximated by DP operators (Theorem \ref{t:c_0}). Our main technical tool is an inequality from Lemma \ref{l:max},
which may be of interest in its own right.

In Section \ref{s:C(K)}, we show that any $\vr$-DP operator from
a symmetric sequence space into a $\sigma$-Dedekind complete $C(K)$ space
can be approximated by DP maps (Theorem \ref{t:seq to C(K)}).

Section \ref{s:on l_1} is devoted to proving that any positive
$\vr$-DP operator from $\ell_p$ into $L_p$ is can be approximated by a DP one
(Theorems \ref{t:ell1} and \ref{t:lq->Lq}).
In Section \ref{s:to L1}, we prove similar approximation results
for operators from a sequence space with a shrinking basis to $L_1$.

In Section \ref{s:counter} we show that, for $1 \leq p < q < \infty$,
and any $\vr > 0$, there exists a positive $\vr$-DP contraction $T : \ell_p \to \ell_q$
so that $\|T - S\| \geq 1/2$ for any DP map $S$ (Proposition \ref{p:main}).
Similar results hold for operators from $\ell_p$ into a certain class of Banach lattices,
including $L_q$ (Proposition \ref{p:lower q}).

Section \ref{s:modulus} deals with the connections
between the properties of an operator and its modulus. We start by observing
that, if $T \in B(E,F)$ is regular, and $|T|$ is $\vr$-DP, then the same holds
for $T$. Under some conditions on $E$ and $F$, the converse is true
(Proposition \ref{p:modulus}). In general, Proposition \ref{p:bad mod}
provides a counterexample.

Finally, in Section \ref{s:lattice p estimates} we explore notions closely related to $\vr$-DP operators, such as almost lattice homomorphisms, and operators almost preserving expressions of the form $(|x|^p+|y|^p)^{1/p}$. Further, we explore the connections between $\vr$-DP operators, and operators ``almost preserving'' order (Proposition \ref{p:connections}). We also consider a stronger version of $\vr$-DP operators for which approximation results holds in a general setting (see Theorem \ref{t:SMP}).

Throughout this paper, we use standard Banach lattice terminology and
notation, as well as some well known facts. For more information we refer the reader
to many the excellent monographs on the topic, such as \cite{A-B} or \cite{M-N}.
For the peculiarities of complex Banach lattices, one may consult \cite{AA}.

\section{Basic facts}\label{s:basics}

We start by a few easy observations.
First, almost disjointness preservation
only needs to be verified on positive elements. More precisely:

\begin{proposition}\label{p:positive}
Suppose $E$ and $F$ are real (complex) Banach lattices. If $T \in B(E,F)$ is such that
$\||Tx| \wedge |Ty|\| \leq \vr$ for any positive disjoint
$x, y \in \ball(E)$, then $T$ is $4 \vr$-DP ($16 \vr$-DP in the complex case).
Moreover, if $T$ is positive then it is $\vr$-DP.
\end{proposition}

\begin{proof}
Suppose first $T$ is positive.
Then, for every $z \in E$, we have $|Tz| \leq T|z|$
(see e.g. \cite[Lemma 3.22]{AA}).
If $x$ and $y$ are disjoint, then
$$
\big\| |Tx| \wedge |Ty| \big\| \leq
\big\| T |x| \wedge T |y| \big\| \leq
\vr.
$$
For general $T$, in the real case, write
$x = x_+ - x_-$, and $y = y_+ - y_-$ (here $x \perp y$).
Then
$$  \eqalign{
\big\||Tx| \wedge |Ty|\big\|
&
\leq
\big\|\big( |Tx_+| + |Tx_-| \big) \wedge \big( |Ty_+| + |Ty_-| \big)\big\| \leq
\sum_{\sigma, \delta = \pm} \big\||Tx_\sigma| \wedge |Ty_\delta|\big\|
\cr
&
\leq
\vr \sum_{\sigma, \delta = \pm} \max\{\|x_\sigma\|, \|y_\delta\|\} \leq
4 \vr.
}  $$
The complex case is dealt with similarly.
\end{proof}

Furthermore, almost disjointness preserving operators also preserve
``almost disjointness'':

\begin{proposition}\label{p:alm disj}
Suppose $E$ and $F$ are real Banach lattices, and $T \in B(E,F)$ is $\vr$-$\DP$.
Then
$$
\||Tx| \wedge |Ty|\| \leq
 4 \big( \vr  \max\{\|x\|, \|y\|\}  + \|T\| \||x| \wedge |y|\| \big)
$$
for any $x, y \in E$. In the complex case, a similar inequality holds,
with $16$ instead of $4$.
\end{proposition}

\begin{proof}
We prove the real case.
Suppose first $x$ and $y$ are positive. Then $x^\prime = x - x \wedge y$
and $y^\prime = y - x \wedge y$ are disjoint, and therefore,
$$
\| |T x^\prime| \wedge |T y^\prime| \| \leq \vr \max\{\|x^\prime\|, \|y^\prime\|\} \leq \vr \max\{\|x\|, \|y\|\} .
$$
However,
\begin{align*}  
\| |T x| \wedge |T y| \|
&
\leq
\| (|T x^\prime| + |T(x \wedge y)|) \wedge (|T y^\prime| + |T(x \wedge y)|) \|
\\
=
&
\| |T x^\prime| \wedge |T y^\prime| + |T(x \wedge y)| \| \leq
\| |T x^\prime| \wedge |T y^\prime| \| + \|T(x \wedge y)\|
\\
\leq
&
\vr \max\{\|x\|, \|y\|\}+ \|T\| \|x \wedge y\| .
\end{align*}   

For general $x,y \in E$, use the Riesz decompositions $x = x_+ - x_-$ and $y = y_+ - y_-$.
For $\sigma, \delta = \pm$, we have $x_\sigma \wedge y_\delta \leq |x| \wedge |y|$,
hence $\|x_\sigma \wedge y_\delta\| \leq \||x| \wedge |y|\|$.

By the above,
\begin{align*}
\| |T x_\sigma| \wedge |T y_\delta| \|
&
\leq
 \vr \max\{\|x_\sigma\|, \|y_\delta\|\} +
 \|T\| \|x_\sigma \wedge y_\delta\|
\\
&
\leq
 \vr \max\{\|x\|, \|y\|\} + \|T\| \| |x| \wedge |y| \| .
\end{align*}
To finish the proof, recall that
$|Tx| \wedge |Ty| \leq \sum_{\sigma, \delta = \pm} |T x_\sigma| \wedge |T y_\delta|$.
\end{proof}

Finally, we show that, if a Banach lattice $E$ is ``diffuse enough'',
and $F$ is ``atomic enough'', then the norm of a $\vr$-DP operator
from $E$ to $F$ cannot exceed $2\vr$.
We say that a Banach lattice $E$ has \emph{Fatou norm with constant $\fatou$}
if, for any non-negative increasing net $(x_i) \subset E$, with
$\sup_i \|x_i\| < \infty$, we have $\vee_i x_i \in E$, and
$\| \vee_i x_i \| \leq {\fatou} \sup_i \|x_i\|$. Recall that $x \in E_+ \backslash \{0\}$ is called an \emph{atom} of $E$
if it generates a one-dimensional principal ideal $E_x$.
In this case, $E_x$ is actually a projection band \cite[Proposition 4.18]{Schw}.
Moreover, $x$ is an atom if and only if whenever $0 \leq x_1, x_2 \leq x$, and
$x_1 \perp x_2$, then either $x_1 = 0$ or $x_2 = 0$.
A Banach lattice is called \emph{atomic} if it is generated by its atoms
as a band (see e.g. \cite[Section 2.5]{M-N}).

\begin{proposition}\label{p:->atomic}
Suppose $E$ and $F$ are Banach lattices, so that
$E$ is order continuous and has no atoms, while
$F$ is atomic, and has Fatou norm with constant $\fatou$.
If $T : E \to F$ is $\vr$-DP, then $\|T\| \leq 2\vr \fatou$.
\end{proposition}

The restriction on $E$ being order continuous is essential.
For instance, suppose $E = C(K)$, and $F$ is $1$-dimensional.
Then any scalar multiple of a point evaluation is a DP functional
(see \cite{Do02} for the proof that any $\vr$-DP functional is close
to a scalar multiple of a point evaluation).

\begin{proof}
Denote the atoms of $F$ by $(\delta_i)_{i \in I}$.
By the discussion above, for every $i\in I$, $\spn[\delta_i]$
is the range of a band projection. We denote this band
projection by $P_i$, and write $P_i x = \langle f_i, x \rangle \delta_i$,
where $f_i \in F^*_+$. For a finite set $A \subset I$, define the ``basis''
projection $Q_A = \sum_{i \in A} P_i$. It is easy to see
(cf. \cite[pp.~142-144]{Scha}) that, for any $y \in F$, the net
$(Q_A y)$ converges to $y$ in the order topology (here, the net
of finite subsets of $I$ is ordered by inclusion).

Fix $c < \|T\|$, and find $x \in E$ so that $\|x\| \leq 1$, and $\|Tx\| > c$.
Further, find a finite set $A$ so that $\|Q_A T x\| > c/\fatou$.
Let $P_x$ be the band projection corresponding to $|x|$, and denote
its image by $G$. Note that $G$ inherits the lack of atoms from $E$.
Indeed, suppose, for the sake of contradiction, that $y \in G_+$ is
an atom of $G$. By \cite[Lemma 2.7.12]{M-N}, there exist non-zero disjoint
$y_1, y_2 \in E_+$ so that $y = y_1 + y_2$. By the properties
of band projections, $y_1, y_2 \in G$.

By \cite[Theorem 1.b.4]{LT2}, we can view $G$ as a K\"othe
function space on $(\Omega, \mu)$. The proof (in conjunction with
the characterization of atoms given above) actually constructs
a measure $\mu$ without atoms. Moreover, there exist $\mu$-measurable
functions $\phi_i$ so that, for every $y \in G$,
$\langle f_i, Ty \rangle = \int_\Omega \phi_i y \, d\mu$.
By Liapounoff's Theorem (see e.g. \cite[Theorem 2.c.9]{LT2}),
there exists a subset $S \subset \Omega$ so that the equality
$$
\langle f_i, T(x \one_S) \rangle = \langle f_i, T(x \one_{S^c}) \rangle =
 \frac{\langle f_i, Tx \rangle}{2}
$$
holds for any $i \in A$. As $Q_A$ is a band projection, we have, for every $z \in F$,
$$
Q_A |z| = |Q_A z| = \sum_{i \in A} \big| \langle f_i , z \rangle \big| \delta_i .
$$
Consequently,
$$
Q_A |Tx| = \sum_{i \in A} \big| \langle f_i , Tx \rangle \big| \delta_i =
2 Q_A |T(x\one_S)| = 2 Q_A |T(x\one_{S^c})| ,
$$
hence
$$
\big\| |T(x\one_S)| \wedge |T(x\one_{S^c})| \big\| \geq
\frac12 \big\| Q_A |Tx| \big\| > \frac{c}{2\fatou} .
$$
However, $x \one_S$ and $x \one_{S^c}$ belong to $\ball(X)$, hence
$\big\| |T(x\one_S)| \wedge |T(x\one_{S^c})| \big\| \leq \vr$.
To complete the proof, recall that $c$ can be arbitrarily close to $\|T\|$.
\end{proof}

\section{A probabilistic inequality}\label{s:ineq}

The following lemma may be interesting in its own right.

\begin{lemma}\label{l:max}
Suppose $(b_i)_{i=0}^n$ is a family of non-negative numbers.
Then
$$
\expe_S \min \big\{ \sum_{i \in S} b_i ,
 \sum_{i \in S^c}  b_i \big\} \leq
 \big( \sum_{i=0}^n b_i - \max_{0 \leq i \leq n} b_i \big) \leq 2^{8}\,\expe_S \min \big\{ \sum_{i \in S} b_i ,
 \sum_{i \in S^c}  b_i \big\}.
$$
Here, the expected value is taken over all subsets $S \subset \{0, \ldots, n\}$,
with equal weight.
\end{lemma}

\begin{proof}
Clearly, for every $S\subset\{0, \ldots, n\}$ we have
$$
\min \big\{ \sum_{i \in S} b_i , \sum_{i \in S^c}  b_i \big\} \leq \sum_{i=0}^n b_i - \max_{0 \leq i \leq n} b_i
$$
and therefore, the first inequality of the claim follows.

For the second one, without loss of generality, we can assume
$1 = b_0 \geq b_1 \geq \ldots \geq b_n \geq 0$, and set
$b = b_1 + \ldots + b_n$. For $S \subset \{0, \ldots, N\}$,
let $f(S) = \sum_{i \in S} b_i$ and $g(S) = \min\{f(S), f(S^c)\}$.

Consider two cases.

(1) $b \leq 2^7$. For $S \subset \{0, \ldots, N\}$ set $S^\prime = S$ if $0 \notin S$, and
$S^\prime = S^c$ otherwise. Then $S^\prime$ is uniformly distributed
over subsets of $\{1, \ldots, \, n\}$. Then
$$
2^{-7} \sum_{i \in S^\prime} b_i \leq 2^{-7} b \leq 1 \leq
\sum_{i \in \{0, \ldots, N\} \backslash S^\prime} b_i ,
$$
hence $g(S) \geq 2^{-7} \sum_{i \in S^\prime} b_i$. Note that $S^\prime$
is uniformly distributed over subsets of $\{1, \ldots, \, n\}$, hence
$$
\expe_S g(S) \geq 2^{-7} \expe_{S^\prime \subset \{1, \ldots, N\}} \sum_{i\in S^\prime} b_i =
2^{-7} \cdot \frac{b}{2} = 2^{-8} b .
$$

(2) $b > 2^7$. Note that $\sum_{i=0}^n b_i^2 \leq \sum_{i=0}^n b_i = b+1$.
By the large deviation inequality for Bernoulli random variables
(see e.g. \cite[Chapter 7]{MS}),
$$ \eqalign{
\prob \Big( \big| b + 1 - 2 \sum_{i \in S} b_i \big| \geq (b+1)/4 \Big)
&
\leq
2 \exp \big( - ((b+1)/4)^2 / (4(b+1)) \big)
\cr
&
=
2 e^{-(b+1)/64} < 2 e^{-1} < 0.74 .
} $$
Thus, with probability greater than $0.26$,
$$
\sum_{i \in S} b_i \in \Big[ \frac{b+1}{4} , \frac{3(b+1)}{4} \Big] ,
$$
hence $g(S) \geq (b+1)/4$. Therefore,
$$
\expe g(S) \geq 0.26 \Big( \frac{b+1}{4} \Big) > 2^{-5} b .
$$
Thus, each of the cases gives the desired result.
\end{proof}

Now an application of Krivine functional calculus (cf. \cite[Theorem 1.d.1]{LT2}) yields:

\begin{corollary}\label{c:max Lq}
If $f_1, \ldots, f_n$ are positive elements in a Banach lattice, then
$$
\expe_S \min \big\{ \sum_{i \in S} f_i , \sum_{i \in S^c}  f_i \big\} \geq
2^{-8} \Big( \sum_{i=1}^n f_i - \vee_{1 \leq i \leq n} f_i \Big) .
$$
Consequently,
$$
\expe_S \Big\| \min \big\{ \sum_{i \in S} f_i , \sum_{i \in S^c}  f_i \big\} \Big\| \geq
2^{-8} \Big\| \sum_{i=1}^n f_i - \vee_{1 \leq i \leq n} f_i \Big\| .
$$
\end{corollary}

As a consequence, we have:

\begin{corollary}\label{c:arb number}
Suppose $T:E\rightarrow F$ is a positive operator which is $\vr$-DP.
Then, for any disjoint $x_1, \ldots, x_n \in E$, we have
$$
\Big\| \sum_{i=1}^n |Tx_i| - \bigvee_{i=1}^n |Tx_i| \Big\| \leq 256 \vr \Big\|\sum_{i=1}^n x_i\Big\|.
$$
In particular, for any disjoint $x_1, \ldots, x_n \in E$ and every $1\leq p<\infty$ it also holds that
$$
\Big\| \Big(\sum_{i=1}^n |Tx_i|^p\Big)^{\frac1p} - T\Big(\sum_{i=1}^n |x_i|^p\Big)^{\frac1p} \Big\| \leq 256 \vr\Big\|\sum_{i=1}^n x_i\Big\|.
$$
\end{corollary}

\begin{proof}
For any $S \subset \{1, \ldots, n\}$, we have
$$
\Big\| \big( \sum_{i \in S} |Tx_i| \big) \bigwedge
 \big( \sum_{i \in S^c} |Tx_i| \big) \Big\| \leq \Big\| T\big| \sum_{i \in S} x_i \big| \bigwedge
 T\big| \sum_{i \in S^c} x_i \big| \Big\| \leq \vr \|\sum_{i=1}^n x_i\| .
$$
Now apply Corollary \ref{c:max Lq}, with $f_i = Tx_i$.

For the second inequality, note that for every $1\leq p<\infty$ we have
$$
0\leq\Big(\sum_{i=1}^n |Tx_i|^p\Big)^{\frac1p} - T\Big(\sum_{i=1}^n |x_i|^p\Big)^{\frac1p}\leq \sum_{i=1}^n |Tx_i| - \bigvee_{i=1}^n |Tx_i|
$$
\end{proof}

\begin{corollary}\label{c:arb max min}
Suppose the operator $T \in B(E,F)_+$ is $\vr$-DP,
and $E$ is $\sigma$-Dedekind complete.
Then, for any $x_1, \ldots, x_n \in E_+$, we have
$$
\max \Big\{ \big\|T(\vee_{i=1}^n x_i) - \vee_{i=1}^n (Tx_i) \big\| ,
 \big\|\wedge_{i=1}^n (Tx_i) - T (\wedge_{i=1}^n x_i) \big\| \Big\}
 \leq 256 \vr \big\|\vee_{i=1}^n x_i\big\| .
$$
\end{corollary}

\begin{proof}
First prove that
\begin{equation}
\label{eq:sum to max}
\big\|T(\vee_{i=1}^n x_i) - \vee_{i=1}^n (Tx_i) \big\|
 \leq 256 \vr \big\|\vee_{i=1}^n x_i\big\| .
\end{equation}
Fix $c > 0$, and let $x = x_1 + \ldots + x_n$.
Let ${\mathcal{C}}$ be the set of
\emph{components} of $x$ -- that is, of vectors $y \in E_+$
satisfying $y \wedge (x - y) = 0$. By \cite[Theorem 1.49]{A-B},
${\mathcal{C}}$ is closed under the operations $\vee$ and $\wedge$.
Moreover, if $u, v \in {\mathcal{C}}$ are such that $u \leq v$,
then $v - u \in {\mathcal{C}}$.
Finite linear combinations of elements of ${\mathcal{C}}$ are called
\emph{simple functions}.

By \cite[Proposition 1.2.20]{M-N}, $E$ has the Principal Projection property.
By Freudenthal Spectral Theorem (see e.g. \cite[Theorem 2.8]{A-B},) for every $i$
there exists a simple function
$u_i$ so that $0 \leq x_i - u_i \leq c |x|/\|x\|$ (hence $\|u_i - x_i\| \leq c$).
By considering $u_i \vee 0$ instead of $u_i$, we can assume that all the $u_i$'s
are non-negative. Write $u_i = \sum_{j=1}^{N_i} \alpha_{ij} v_{ij}$, where
$\alpha_{ij} > 0$ and $(v_{ij})_{j=1}^{N_i}$ are disjoint components of $x$.
By the discussion above, the elements $\wedge_{i=1}^n v_{i j_i}$ for any $j_i\leq N_i$ are
disjoint components of $x$, and therefore, there exists a family $(w_j)_{j=1}^M$
of disjoint components of $x$, so that for each $i$ we can write
$u_i = \sum_{j=1}^M \beta_{ij} w_j$. Note that $\vee_i u_i = \sum_j \beta_j w_j$,
where $\beta_j = \vee_i \beta_{ij}$.

Define the sets $(A_i)$ recursively by setting $A_0 = \emptyset$, and
$A_i = \{j : \beta_{ij} = \beta_j\} \backslash \cup_{s < i} A_s$.
These sets are clearly disjoint, and their union is $\{1, \ldots, M\}$.
For $1 \leq i \leq n$ set $y_i = \sum_{j \in A_i} \beta_j w_j$.
Then $0 \leq y_i \leq u_i$, the $y_i$'s are disjoint, and
$\vee_i y_i = \vee_i u_i \leq \vee_i x_i$.
Thus,
$$
T(\vee_{i=1}^n u_i) - \vee_{i=1}^n (Tu_i) \leq
T(\vee_{i=1}^n y_i) - \vee_{i=1}^n (Ty_i) =
T(\sum_{i=1}^n y_i) - \vee_{i=1}^n (Ty_i) .
$$
By Corollary \ref{c:arb number},
\begin{equation}
\begin{aligned}
\label{eq:vee u_i y_i}
\big\|T(\vee_{i=1}^n u_i) - \vee_{i=1}^n (Tu_i)\big\|
&
\leq
\big\|T(\sum_{i=1}^n y_i) - \vee_{i=1}^n (Ty_i)\big\|
\\
&
\leq
256 \vr \big\| \sum_{i=1}^n y_i \big\| \leq 256 \vr \big\| \vee_{i=1}^n x_i \big\| .
\end{aligned}
\end{equation}

For each $i$ write $x_i = u_i + z_i$, where $z_i \geq 0$, and $\|z_i\| \leq c$.

In this notation,
$$
\bigvee_{i=1}^n x_i \leq \bigvee_{i=1}^n u_i+\bigvee_{i=1}^n z_i,
$$
and therefore,
$$
\Big\| \bigvee_{i=1}^n x_i - \bigvee_{i=1}^n u_i \Big\| \leq
 n c.
$$
From this, we conclude that
$$
\big\|T(\vee_{i=1}^n x_i) - \vee_{i=1}^n (Tx_i) \big\| \leq
\big\|T(\vee_{i=1}^n u_i) - \vee_{i=1}^n (Tu_i) \big\| +
n c \|T\| .
$$
To obtain \eqref{eq:sum to max}, invoke \eqref{eq:vee u_i y_i},
and recall that $c$ can be arbitrarily small.

To obtain the inequality
\begin{equation}
\label{eq:sum to min}
\big\|\wedge_{i=1}^n (Tx_i) - T (\wedge_{i=1}^n x_i) \big\|
 \leq 256 \vr \big\|\vee_{i=1}^n x_i\big\| ,
\end{equation}
set $x = \vee_{i=1}^n x_i$. For each $i$ set $y_i = x - x_i$,
then $0 \leq y_i \leq x$. We have
$\vee_{i=1}^n y_i = x + \vee_{i=1}^n (y_i - x) = x - \wedge_{i=1}^n x_i$,
hence $T (\wedge_{i=1}^n x_i) = Tx - T(\vee_{i=1}^n y_i)$. Similarly,
$\vee_{i=1}^n T y_i = T x + \vee_{i=1}^n (T(y_i - x)) = T x - \wedge_{i=1}^n T x_i$,
which yields $\wedge_{i=1}^n T x_i = Tx - \vee_{i=1}^n T y_i$.
Therefore,
$$
\wedge_{i=1}^n (Tx_i) - T (\wedge_{i=1}^n x_i) =
T(\vee_{i=1}^n y_i) - \vee_{i=1}^n (Ty_i) .
$$
To obtain \eqref{eq:sum to min}, combine \eqref{eq:sum to max}
with the fact that $\vee_{i=1}^n y_i \leq x$.
\end{proof}

It was shown in \cite{Abr} that for any r.i. spaces $X, Y$ over a finite measure such that $X\nsubseteq Y$ there is no non-zero disjointness preserving operator $T:X\rightarrow Y$. In particular, the only disjointness preserving operator $T:L_p[0,1]\rightarrow L_q[0,1]$ for $p>q$ is $T=0$. An application of Corollary \ref{c:arb number} provides the following version of this fact for positive $\vr$-DP operators:

\begin{proposition}\label{p:norm Lp}
Let $1\leq p<q\leq\infty$ and $E$ be a $q$-convex Banach lattice. If $T:L_p[0,1]\rightarrow E$ is positive, $\vr$-DP , then $\|T\|\leq 256\vr$.
\end{proposition}

\begin{proof}
Given a positive $x\in L_p[0,1]$ with $\|x\|_p=1$, for every $n\in\mathbb N$, an application of Liapunov's theorem \cite[Theorem 2.c.9]{LT2} allows us to find a partition of $[0,1]$ in pairwise disjoint measurable sets $(A_i)_{i=1}^n$ such that $\|x\chi_{A_i}\|_p=n^{-1/p}$. Let $x_i=x\chi_{A_i}$, for $i=1,\ldots, n$. We have that $(x_i)_{i=1}^n$ are disjoint and $x=\sum_{i=1}^n x_i$.

Since $E$ is $q$-convex, there is a constant $C>0$ so that
$$
\Big\|\Big(\sum_{i=1}^n |Tx_i|^q\Big)^{\frac1q}\Big\|\leq C\Big(\sum_{i=1}^n \|Tx_i\|^q\Big)^{\frac1q}\leq C\|T\| n^{\frac1q-\frac1p}.
$$
Hence, using Corollary \ref{c:arb number}, we have
\begin{align*}
\|Tx\|&\leq \Big\|T\Big(\sum_{i=1}^n x_i\Big)-\Big(\sum_{i=1}^n |Tx_i|^q\Big)^{\frac1q}\Big\|+\Big\|\Big(\sum_{i=1}^n |Tx_i|^q\Big)^{\frac1q}\Big\|\\
&=\Big\|T\Big(\sum_{i=1}^n |x_i|^q\Big)^{\frac1q}-\Big(\sum_{i=1}^n |Tx_i|^q\Big)^{\frac1q}\Big\|+\Big\|\Big(\sum_{i=1}^n |Tx_i|^q\Big)^{\frac1q}\Big\|\\
&\leq 256\vr+C\|T\| n^{\frac1q-\frac1p}.
\end{align*}
Since $p<q$ and $n$ was arbitrary, we get that $\|T\|\leq 256\vr$.
\end{proof}

\section{Positive operators on $\ell_\infty^n$, $c_0$ and $c$}\label{s:c_0}

Recall that a Banach lattice $X$ has \emph{Fatou Property with constant $\fatou$}
if, for any non-negative increasing net $(x_i) \subset X$, with
$\sup_i \|x_i\| < \infty$, we have $\vee_i x_i \in X$, and
$\| \vee_i x_i \| \leq {\fatou} \sup_i \|x_i\|$.
If $\fatou = 1$, we speak simply of the \emph{Fatou Property}. Every Banach lattice with the Fatou property is $\sigma$-Dedekind complete. Note that, if $X$ is a K\"othe function space, then it suffices to
verify the above inequality for non-negative increasing sequences $(x_i)$.
By \cite[Proposition 2.4.19]{M-N}, any dual Banach lattice has the Fatou
Property. Also, by \cite[Section 1.a]{LT2}, any order continuous Banach lattice
has the Fatou Property.

\begin{theorem}\label{t:c_0}
Suppose $F$ is a Banach lattice, and consider $\vr > 0$.
\begin{enumerate}
\item
For any positive operator $T : \ell_\infty^n \to F$, which is $\vr$-DP,
there exists a DP operator $S : \ell_\infty^n \to F$, so that $0 \leq S \leq T$,
and $\|T - S\| \leq 256 \vr$.
\item
Suppose $F$ has the Fatou Property with constant $\fatou$, then for any
positive operator $T : c_0 \to F$, which is $\vr$-DP,
there exists a DP operator $S : c_0 \to F$, so that $0 \leq S \leq T$,
and $\|T - S\| \leq 256 \fatou \vr$.
\item
Suppose $F$ has the Fatou Property with constant $\fatou$, then for any
positive operator $T : c \to F$, which is $\vr$-DP,
there exists a DP operator $S : c \to F$, so that $0 \leq S \leq T$,
and $\|T - S\| \leq 256 \fatou^2 \vr$.
\end{enumerate}
\end{theorem}

The following lemma is needed to prove Theorem \ref{t:c_0}. This result may be known to the experts, but we haven't been able
to find it in the literature.

\begin{lemma}\label{l:sup}
Suppose that for $1\leq i\leq k$, $(x_n^{(i)})_{n\in\mathbb N}$ are increasing positive sequences in a
Banach lattice, so that $\vee_{n\in\mathbb N} x_n^{(i)}$ for $1\leq i\leq k$
and $\bigvee_{n\in \mathbb N} \Big(\sum_{i=1}^k x_n^{(i)}\Big)$ exist. Then
$$
\bigvee_{n\in \mathbb N} \Big(\sum_{i=1}^k x_n^{(i)}\Big) = \sum_{i=1}^k \bigvee_{n\in \mathbb N} x_n^{(i)}.
$$
\end{lemma}

\begin{proof}
We will proceed by induction on $k$. For any $m \in \NN$, we have
$$
\bigvee_{n\in \mathbb N} \Big(\sum_{i=1}^{k+1} x_n^{(i)}\Big)\geq \bigvee_{n\in \mathbb N} \Big(\sum_{i=1}^{k} x_n^{(i)}+x_m^{(k+1)}\Big)= \bigvee_{n\in \mathbb N} \Big(\sum_{i=1}^{k} x_n^{(i)}\Big)+x_m^{(k+1)},
$$

hence, using the induction hypothesis,
$$
\bigvee_{n\in \mathbb N} \Big(\sum_{i=1}^{k+1} x_n^{(i)}\Big)\geq \bigvee_{n\in \mathbb N} \Big(\sum_{i=1}^{k} x_n^{(i)}\Big) + \bigvee_{m\in\mathbb N} x_m^{(k+1)}=\sum_{i=1}^{k+1} \bigvee_{n\in \mathbb N} x_n^{(i)}.
$$

The converse inequality follows from the fact that, for every $m$,
$$
\bigvee_{n=1}^m (\sum_{i=1}^{k} x_n^{(i)} ) = \sum_{i=1}^{k} x_m^{(i)}  \leq \sum_{i=1}^k \bigvee_{n\in \mathbb N} x_n^{(i)}.
$$
\end{proof}

\begin{proof}[Proof of Theorem \ref{t:c_0}]
Throughout the proof, we denote by $(\delta_i)$ the canonical basis of
$\ell_\infty^n$ or $c_0$, and $f_i=T\delta_i$. Furthermore, we assume that $\|T\| \leq 1$.
Indeed, if $\|T\| > 1$, then $T^\prime = T/\|T\|$ is $\vr/\|T\|$-DP.
If (1) is established for a contractive operator $T$, then we can find
a DP map $S^\prime$ so that $0 \leq S^\prime \leq T^\prime$, and
$\|S^\prime - T^\prime\| \leq 256 \vr/\|T\|$ and take $S=\|T\|S'$. For (2) and (3) the same argument works.

Let us start by defining for each $n\in\mathbb N$ a function $\phi_n : \RR^n \to \RR$ given by
$$
\phi_n : (t_1, \ldots, t_n) \mapsto
\left\{ \begin{array}{ll}
   0                             &   t_1 \leq \vee_{i=2}^n |t_i|      \\
   2(t_1 - \vee_{i=2}^n |t_i|)
       &   \vee_{i=2}^n |t_i| \leq t_1 \leq 2 \vee_{i=2}^n |t_i|      \\
   t_1                           &   t_1 >2 \vee_{i=2}^n |t_i|
\end{array} \right. .
$$
\medskip

(1) For $1 \leq i \leq n$ set
$$
g_i = \phi_n(f_i, f_{i+1}, \ldots, f_n, f_1, \ldots, f_{i-1}).
$$
We claim that the operator $S : \ell_\infty^n \to F : \delta_i \mapsto g_i$
has the desired properties.

Note that $0 \leq \phi_n(t_1, \ldots, t_n) \leq t_1$, hence $0 \leq g_i \leq f_i$,
which shows that $0 \leq S \leq T$.

To show that $S$ is disjointness preserving, consider $i \neq j$. Note that, for any
$(t_1, \ldots, t_n) \in \RR^n$,
$$
\phi_n(t_i, t_{i+1}, \ldots, t_n, t_1, \ldots, t_{i-1}) \wedge
\phi_n(t_j, t_{j+1}, \ldots, t_n, t_1, \ldots, t_{j-1}) = 0 ,
$$
hence $g_i$ and $g_j$ are disjoint.

Finally we estimate
$$
\|T-S\| = \big\| (T-S) \sum_{i=1}^n \delta_i \big\| =
 \big\| \sum_{i=1}^n (f_i - g_i) \big\| .
$$
We claim that
\begin{equation}
\label{eq:sum}
\sum_{i=1}^n (f_i - g_i) \leq
 2^9 \expe_S \Big( \sum_{i \in S} f_i \Big) \wedge \Big( \sum_{i \in S^c} f_i \Big) .
\end{equation}
Indeed, by functional calculus,
we need to show that, for any $t_1, \ldots, t_n \in \RR^n$,
$$  \eqalign{
\sum_{i=1}^n \big(t_i - \phi_n(t_i, t_{i+1}, \ldots, t_n, t_1, \ldots, t_{i-1})\big)
\cr
&
\leq
2^9 \expe_S \Big( \sum_{i \in S} t_i \Big) \wedge \Big( \sum_{i \in S^c} t_i \Big) .
}  $$
By relabeling, we can assume that $t_1 \geq t_2 \geq \ldots \geq t_n$.
By Lemma \ref{l:max}, the right hand side is at least
$2(t_2 + \ldots + t_n)$. In the left hand side however,
$$
t_2 - \phi_n(t_2, t_3, \ldots, t_n, t_1) = t_2, \ldots,
t_n - \phi_n(t_n, t_1, \ldots, t_{n-1}) = t_n ,
$$
while
$$
0 \leq t_1 - \phi_n(t_1, t_2, \ldots, t_n) \leq \vee_{i \geq 2} t_i = t_2 .
$$
Therefore, the right hand side is at most
$2 t_2 + t_3 + \ldots + t_n \leq 2(t_2 + \ldots + t_n)$. Finally, since $T$ is $\vr$-DP, the result follows.

(2) For $ T : c_0 \to F$, let $f_i = T \delta_i$. For $n \geq i$, set
$$
g_i^{(n)} = \phi_n(f_i, f_1, \ldots, f_{i-1}, f_{i+1}, \ldots, f_n) .
$$
Clearly, $0 \leq g_i^{(n)} \leq f_i$. Moreover,
it is easy to observe that
$$
\phi_{n}(t_1, \ldots, t_n)=\phi_{n+1}(t_1, \ldots, t_n, 0) \geq \phi_{n+1}(t_1, \ldots, t_n, t_{n+1})
$$
for any $t_{n+1} \in \RR$. As the Krivine functional calculus preserves
lattice operations, we have
$$   \eqalign{
g_i^{(n)}
&
=
\phi_{n+1}(f_i, f_1, \ldots, f_{i-1}, f_{i+1}, \ldots, f_n, 0)
\cr
&
\geq
\phi_{n+1}(f_i, f_1, \ldots, f_{i-1}, f_{i+1}, \ldots, f_n, f_{n+1}) =g_i^{(n+1)} ,
}   $$
hence the sequence $(g_i^{(n)})_n$ is decreasing, for every $i$.
Due to the $\sigma$-Dedekind completeness of $F$,
$g_i = \wedge_n g_i^{(n)}$ exists in $F_+$.
Define the operator $S : c_0 \to F$ by $S \delta_i = g_i$.

Clearly $0 \leq S \leq T$. Moreover, $g_i^{(n)} \wedge g_j^{(n)}=0$ whenever
$i, j \in \{1, \ldots, n\}$ are distinct, hence $g_i \perp g_j$ for $i \neq j$,
and consequently, $S$ is disjointness preserving. Moreover,
$$
\|T-S\| = \sup_n \big\| (T-S) \sum_{i=1}^n \delta_i \big\| =
\sup_n \big\| \sum_{i=1}^n (f_i - g_i) \big\| .
$$
Reasoning as in (1), we conclude that, for every $k \geq n$,
$$
\big\| \sum_{i=1}^n (f_i - g_i^{(k)}) \big\| \leq
\big\| \sum_{i=1}^k (f_i - g_i^{(k)}) \big\| \leq 256 \vr .
$$
By the Fatou Property and Lemma \ref{l:sup},
$$
\big\| \sum_{i=1}^n (f_i - g_i) \big\| =
\big\| \vee_{k=1}^\infty \sum_{i=1}^n (f_i - g_i^{(k)}) \big\| \leq 256 \fatou \vr .
$$

(3) As before, let $(\delta_i)$
be the canonical basis of $c_0 \subset c$, and denote by $\one$ the constant
sequence $(1,1,\ldots) \in c$. Let $f_i = T \delta_i$, and
$$
f_0 = T \one - \vee_{n=1}^\infty \big(\sum_{i=1}^n f_i \big) .
$$
Note that $\sum_{i=1}^n f_i  = T(\sum_{i=1}^n \delta_i) \leq T \one$,
hence the supremum in the centered equation exists, due to the
$\sigma$-Dedekind completeness of $F$. Note also that, for
$x = (\alpha_1, \alpha_2, \ldots) \in c$,
$$
Tx = (\lim_j \alpha_j) T \one +
 \sum_{i=1}^\infty \big( \alpha_i - \lim_j \alpha_j \big) f_i .
$$
Further observe that, for any $S \subset \{0,1, \ldots, n\}$, we have
$$
\Big\| \big( \sum_{i \in S} f_i \big) \wedge
 \big( \sum_{i \in S^c} f_i \big) \Big\| \leq \vr
$$
(here $S^c = \{0, 1, \ldots, n \} \backslash S$).
Indeed, suppose without loss of generality that $0 \in S$.
Let $S^\prime = S \backslash \{0\}$,
$y = \sum_{i \in S^c} \delta_i$, and $x = \one - y$.
As $T$ is $\vr$-DP, $\|Tx \wedge Ty\| \leq \vr$.
But $Ty = \sum_{i \in S^c} f_i$, while
$$
Tx = \sum_{i \in S^\prime} f_i + T \one - \sum_{i=1}^n f_i \geq
\sum_{i \in S^\prime} f_i + T \one - \vee_{m=1}^{\infty} \sum_{i=1}^m f_i  =
\sum_{i \in S^\prime} f_i + f_0 = \sum_{i \in S} f_i .
$$
Define $g_i^{(n)} =\phi_{n+1}(f_i, f_0, \ldots, f_{i-1}, f_{i+1}, \ldots, f_n)$, for $0 \leq i \leq n$.
As in the proof of (2),
$$
\big\| \sum_{i=0}^n (f_i - g_i^{(n)}) \big\| \leq 256 \vr .
$$
Let $g_i = \lim_k g_i^{(k)}$, then
$$
\big\| \sum_{i=0}^n (f_i - g_i) \big\| \leq 256 \fatou \vr
$$
for every $n$.

Now observe that $g_i^{(i)} \geq g_i^{(i+1)} \geq \ldots$, and
set $\tilde{g} = \vee_{n=1}^\infty \sum_{i=1}^n g_i$.
Define $S : c \to F$ by setting $S \delta_i = g_i$, and
$S \one = \tilde{g} + g_0$. This operator is well-defined and positive.
Moreover, $(T - S) \delta_i = f_i - g_i$ for $i \in \NN$, and,
by Lemma \ref{l:sup},
$$
(T-S) \one = \vee_{n=0}^\infty \sum_{i=0}^n (f_i - g_i) .
$$
Thus, $T \geq S$. Indeed, suppose $x = (\alpha_i)_{i=1}^\infty \in c$ is positive.
Let $\alpha = \lim_j \alpha_j$, then
$$
(T - S) x = \alpha \big( \vee_{n=0}^\infty \sum_{i=0}^n(f_i - g_i) \big) +
\sum_{i=1}^\infty (\alpha_i - \alpha) (f_i - g_i) \geq \alpha( f_0 - g_0 )\geq 0 .
$$
Consequently,
$$
\|T-S\| = \|(T-S) \one\| = \| \vee_{n=0}^\infty \sum_{i=0}^n (f_i - g_i) \|
 \leq \fatou \sup_n \big\| \sum_{i=0}^n (f_i - g_i) \big\| \leq
256 \fatou^2 \vr .
$$ \end{proof}

\section{Operators into $C(K)$ spaces}\label{s:C(K)}

In this section we consider operators from sequences spaces into $C(K)$.
Throughout the section, $K$ denotes a compact Hausdorff space.
First, consider the case when $C(K)$ is $\sigma$-Dedekind complete
(equivalently, $K$ is a basically disconnected compact Hausdorff set,
see \cite[Proposition 1.a.4]{LT2}).

\begin{theorem}\label{t:seq to C(K)}
Suppose $X$ is a Banach lattice with the order structure given by its
$1$-unconditional basis, and $C(K)$ is $\sigma$-Dedekind complete.
If $T : X \to C(K)$ is $\vr$-DP, then there exists a
disjointness preserving $S : X \to C(K)$ so that
$\|S\| \leq \|T\|$, and $\|S - T\| \leq 257 \vr \|T\|$.
If $T$ is positive, then $S$ can be chosen so that, in addition,
$0 \leq S \leq T$. 
\end{theorem}

\begin{proof}
By scaling, we can assume that $T$ is a contraction.
Denote the canonical normalized basis of $X$ by $(\delta_i)_{i=1}^\infty$,
and let $c_{00}$ be the linear span of $\delta_1, \delta_2, \ldots$ in $X$.
For $i \in \NN$, set $f_i = T \delta_i$, and note that $|f_i| \leq \one$.
Consequently, the sequence $(f_i)$ is order bounded, hence, by the
$\sigma$-Dedekind completeness of $C(K)$, $h_i = \vee_{j \neq i} |f_j|$
is continuous for every $i$. Let us define the continuous functions
$$
g_i(t) = \left\{ \begin{array}{ll}
  0            &   |f_i(t)| \leq h_i(t)    \\
  f_i(t)       &   |f_i(t)| \geq 2 h_i(t)    \\
  2 \big(f_i(t) - \sign f_i(t) \cdot h_i(t) \big)
     &    h_i(t) \leq |f_i(t)| \leq 2 h_i(t).
\end{array} \right.
$$

Now, let $S : c_{00} \to C(K) : \delta_i \mapsto g_i$. Clearly, $S$ is disjointness preserving since $|g_i|\wedge |g_j|=0$ for $i\neq j$.
It remains to show that $T|_{c_{00}} - S$ is
bounded, and its norm does not exceed $257 \vr$ (once this is done, we
exend $S$ to the whole space $X$ by continuity).

To this end, fix $t \in K$, and pick $\alpha_1, \ldots, \alpha_N \in \FF$
with $\|\sum_{i=1}^N \alpha_i \delta_i\|_X \leq 1$. We have to show that for every $t\in K$
\begin{equation}
\label{eq:257vr}
\sum_{i=1}^N |\alpha_i| |f_i(t) - g_i(t)| \leq 257 \vr.
\end{equation}
It suffices to consider $\alpha_1, \ldots, \alpha_N \geq 0$.

For $S \subset \{1, \ldots, N\}$, set
$S^c = \{1, \ldots, N\} \backslash S$.
Consider $x = \sum_{i \in S} \omega_i \alpha_i \delta_i$ and
$y = \sum_{i \in S^c} \omega_i \alpha_i \delta_i$, where
$\omega_i = |f_i(t)|/f_i(t)$ if $f_i(t)\neq 0$, and $\omega_i = 0$ otherwise.
Note that $x$ and $y$ are disjoint elements of $\ball(X)$.
As $T$ is $\vr$-DP, we have
$$
\Big( \sum_{i \in S} \alpha_i |f_i(t)| \Big) \wedge
\Big( \sum_{i \in S^c} \alpha_i |f_i(t)| \Big) \leq
\big\| |Tx| \wedge |Ty| \big\|
\leq \vr ,
$$
hence, by Lemma \ref{l:max},
$$
\sum_{i=1}^N \alpha_i |f_i(t)| - \vee_{i=1}^N \alpha_i |f_i(t)| \leq 256 \vr .
$$
Pick $k$ so that $\vee_{i=1}^N \alpha_i |f_i(t)| = \alpha_k |f_k(t)|$.
Note that $|f_k(t) - g_k(t)| \leq \vr$. Indeed, this inequality is evident if
$|f_k(t)| \leq \vr$. If $|f_k(t)| > \vr$, note that $|f_j(t)| \leq \vr$
for any $j \neq k$, otherwise we would have
$\| |T \delta_k| \wedge |T \delta_j| \| > \vr$, contradicting the
assumption that $T$ is $\vr$-DP. Thus, if $|f_k(t)| > \vr$, then $h_k(t) \leq \vr$, and we also have
$|f_k(t) - g_k(t)| \leq h_k(t)$.

As $\alpha_k \leq 1$, we have
$$
\sum_{i=1}^N \alpha_i |f_i(t) - g_i(t)| \leq
\sum_{i \neq k} \alpha_i |f_i(t)| + |f_k(t) - g_k(t)|
\leq 256 \vr + \vr ,
$$
establishing \eqref{eq:257vr}.

If $T$ is positive, then we have $0 \leq g_i \leq f_i$,
hence $0 \leq S \leq T$.
\end{proof}

Along the same lines, we prove:

\begin{theorem}\label{t:f.d. to C(K)}
Suppose $X$ is a finite dimensional Banach lattice.
If $T : X \to C(K)$ is $\vr$-DP, then there exists a
disjointness preserving $S : X \to C(K)$ so that
$\|S\| \leq \|T\|$, and $\|S - T\| \leq 256 \vr \|T\|$.
If $T$ is positive, then $S$ can be chosen so that, in addition,
$0 \leq S \leq T$. 
\end{theorem}

\begin{proof}[Sketch of a proof]
It is well known (see e.g. \cite[Corollary 4.20]{Schw})
that $X$ has a basis of atoms, which we denote by $(\delta_i)_{i=1}^N$
($N = \dim X$). Use scaling to assume that $T$ is contractive.
Let $f_i = T \delta_i$ and $h_i = \vee_{j \neq i} |f_j|$.
Define $g_i$ and $S$ as in the proof of Theorem \ref{t:seq to C(K)},
and proceed further in the same manner.
\end{proof}

For operators from $c$ or $c_0$ into $C(K)$, the assumption that
the range is $\sigma$-Dedekind complete is redundant.

\begin{theorem}\label{t:c0->C(K)}
Suppose $K$ is a compact Hausdorff space, and $\vr$ is a positive number.
Then, for any operator $T : c_0 \to C(K)$, $\vr$-DP , there exists
a DP operator $S : c_0 \to C(K)$ so that $\|S\| \leq \|T\|$,
and $\|T-S\| \leq 257 \vr$.
If $T$ is positive, then $S$ can be selected so that $0 \leq S \leq T$.
\end{theorem}

Here and below, we use the notation $(\delta_i)_{i \in \NN}$ for the canonical
basis of $c_0$, while $c_{00}$ denotes the set of all
finitely supported sequences in $c_0$.
The following straightforward observation will be used throughout the proof.

\begin{lemma}\label{l:c0->C(K)}
A linear map $U : c_{00} \to C(K)$ is bounded if and only if
$$
\sup_{t \in K} \sum_{i=1}^\infty \big| [U\delta_i](t) \big|
$$
is finite. If this is the case, then the above expression equals $\|U\|$.
Moreover, $U$ extends by continuity to an operator from $c_0$ into $C(K)$,
of the same norm.
\end{lemma}

\begin{proof}[Proof of Theorem \ref{t:c0->C(K)}]
We know that, if $T$ is $\vr$-DP, then $T/\|T\|$ is $\vr/\|T\|$-DP.
We can therefore assume that $T$ is a contraction, and
restrict our attention to $\vr < 2^{-8}$.
Denote the canonical basis of $c_0$ by $(\delta_i)_{i=1}^\infty$,
and set $f_i = T \delta_i$.
Note that $T$ is $\vr$-DP if, and only if, the inequality
\begin{equation}
\label{eq:c0 disj}
\Big( \sum_{i \in A} \big| f_i(t) \big| \Big) \wedge
\Big( \sum_{i \in B} \big| f_i(t) \big| \Big) \leq \vr
\end{equation}
holds for any $t \in K$, and for any two disjoint sets $A$ and $B$.
Consequently, for any $t \in K$ there exists at most one $i \in \NN$
so that $|f_i(t)| > \vr$.

Consider the function
$$
\phi(t) = \left\{ \begin{array}{ll}
  0                    &   |t| \leq \vr,   \\
  2(|t|-\vr)\sign t    &   \vr \leq |t| \leq 2\vr,   \\
  t                    &   |t| \geq 2\vr.
\end{array} \right.
$$
Let $g_i = \phi \circ f_i$ (that is, $g_i(t) = \phi(f_i(t))$),
and define the operator $S : c_{00} \to C(K) : \delta_i \mapsto g_i$.
As noted above for any $t \in K$ there exists at most one $i \in \NN$
so that $|g_i(t)| \neq 0$, hence the vectors $(g_i)$ are disjoint,
which shows that $S$ is disjointness preserving. Moreover, if $T$
is positive, then for any $i$, $0 \leq S \delta_i = g_i \leq f_i = T \delta_i$,

First show that $S$ is, indeed, a well-defined contraction
(hence it extends by continuity to a contraction $c_0 \to C(K)$).
By Lemma \ref{l:c0->C(K)},  $\sum_{i=1}^\infty |f_i(t)| \leq 1$
for every $t \in K$. By our construction, $|g_i| \leq |f_i|$, hence
$\sum_{i=1}^\infty |g_i(t)| \leq 1$ for every $t$. By Lemma \ref{l:c0->C(K)}
again, $\|S\| \leq 1$.

It remains to estimate
$$
\|T - S\| = \sup_{t \in K} \sum_{i=1}^\infty \big| [ (T-S) \delta_i ](t) \big| =
\sup_{t \in K} \sum_{i=1}^\infty \big| f_i(t) - g_i(t) \big| .
$$
Fix $t \in K$ and $N \in \NN$, and show that
\begin{equation}
\label{eq:c0->C(K)}
\sum_{i=1}^N \big| f_i(t) - g_i(t) \big| \leq 257 \vr .
\end{equation}

To this end, find $k \in \{1, \ldots, N\}$ so that
$|f_k(t)| = \max_{1 \leq i \leq N} |f_i(t)|$. Then
$|f_j(t)| \leq \vr$ (and consequently, $g_j(t) = 0$) for $j \neq k$.
For a set $S \subset \{1, \ldots, N\}$, set $S^c = \{1, \ldots, N\} \backslash S$.
We know that, for any such $S$,
$$
\sum_{i \in S} \big|f_i(t)\big| \wedge \sum_{i \in S^c} \big|f_i(t)\big| \leq \vr .
$$
Indeed, consider $x = \sum_{i \in S} \overline{\sign f_i(t)} \delta_i$, and
$y = \sum_{i \in S^c} \overline{\sign f_i(t)} \delta_i$. The elements $x$ and $y$
belong to the unit ball of $c_0$, and are disjoint. Thus,
$$
\sum_{i \in S} \big|f_i(t)\big| \wedge \sum_{i \in S^c} \big|f_i(t)\big|
\leq \big\| |Tx| \wedge |Ty| \| \leq \vr .
$$
Then
$$
\sum_{i=1}^N \big| f_i(t) - g_i(t) \big| =
 \sum_{j \neq k} |f_j(t)| + |f_k(t) - g_k(t)| .
$$
By Lemma \ref{l:max}, $\sum_{j \neq k} |f_j(t)| \leq 256 \vr$.
Moreover, $\sup_s |\phi(s) - s| = \vr$, hence
$|f_k(t) - g_k(t)| \leq \vr$. This yields \eqref{eq:c0->C(K)}.
\end{proof}

\begin{theorem}\label{t:c->C(K)}
Suppose $K$ is a compact Hausdorff space, and $\vr$ is a positive number.
For any $\vr$-DP operator $T : c \to C(K)$, there exists
a DP operator $S : c \to C(K)$ so that
$\|T-S\| \leq 536 \vr$. If $T$ is positive, then $S$ can be chosen to be
positive as well.
\end{theorem}

Throughout the proof, we identify $c_0$ with its canonical image in $c$,
then $c = \spn[c_0, \one]$. As before, we denote the canonical basis
of $c_0$ by $(\delta_i)_{i \in \NN}$. The following lemma can be
easily verified.

\begin{lemma}\label{l:norm c}
For any operator $V : c \to X$ ($X$ is an arbitrary Banach space),
we have $\|V\| \leq 2 \|V|_{c_0}\| + \|V \one\|$.
\end{lemma}

\begin{proof}
Consider the projection $Q$ from $c$ to $\FF \one$, defined by
$$
Q \big( (\alpha_1, \alpha_2, \ldots) \big) =
\lim_i \alpha_i \one .
$$
Clearly $\|Q\| = 1$, hence $\|I_c - Q\| \leq 2$. Also,
$\ker Q = \ran(I-Q) = c_0$.
We complete the proof by writing $V = VQ + V(I-Q)$.
\end{proof}

We also need a simple fact about complex numbers. Fix $c > 0$. For a complex number $z = |z| e^{\iota \arg z}$, define
$\phi_c(z) = \big(|z| - c\big)_+ e^{\iota \arg z}$.

\begin{lemma}\label{l:complex}
Given $c>0$, for any $z, w \in \CC$, we have
$|\phi_c(z) - \phi_c(w)| \leq |z-w|$.
\end{lemma}

\begin{proof}
By scaling, we may assume $c = 1$. Without loss of generality, $|z| \geq |w|$.

The case of $|w| \leq 1$ is easy: $\phi_c(w)=0$ and by the triangle inequality,
$$
|z-w| \geq |z| - |w| \geq \big(|z|-1\big)_+ = |\phi_c(z) - \phi_c(w)|.
$$
Now, if $|z| \geq |w| > 1$, use the Law of Cosines:
$\big| z - w \big|^2 = a^2 + b^2 - \kappa ab$, where
$a = |z|$, $b = |w|$, and $\kappa = 2 \cos (\arg z - \arg w)$
(note that $-2 \leq \kappa \leq 2$). Similarly,
$\big| \phi(z) - \phi(w) \big|^2 = (a-1)^2 + (b-1)^2 - \kappa (a-1)(b-1)$.
Thus,
$$
\big| z - w \big|^2 - \big| \phi(z) - \phi(w) \big|^2 =
(2-\kappa)(a+b-1) \geq 0 .
$$
\end{proof}

\begin{lemma}\label{l:c disj}
Suppose $K$ is a compact Hausdorff space, and a contraction
$U : c \to C(K)$ is $\sigma$-DP.
Suppose, moreover, that $U|_{c_0}$ is disjointness preserving,
and the functions $f = U \one$ and $f_i = U \delta_i$ are such that
\begin{equation}
\label{eq:DP on c}
{\textrm{If }} i \in \NN {\textrm{ and }} t \in K {\textrm{ are such that }}
|f_i(t)| > \sigma , {\textrm{ then }} |f(t) - f_i(t)| \leq \sigma .
\end{equation}
Then there exists a DP operator $S : c \to C(K)$ so that
$\|U - S\| \leq 11 \sigma$. If $U$ is positive, then
$S$ can be chosen positive as well.
\end{lemma}

\begin{proof}
We shall construct $g, g_1, g_2, \ldots \in C(K)$ so that:
\begin{enumerate}
\item
For any $i$, $\|g_i - f_i\| \leq 4\sigma$.
\item
$\|g - f\| \leq 3\sigma$.
\item
The functions $g_1, g_2, \ldots$ are disjoint; if $i$ and $t$
are such that $g_i(t) \neq 0$, then $g_i(t) = g(t)$.
\item
If the functions $f, f_1, f_2, \ldots$ are positive, then the same
holds for $g, g_1, g_2, \ldots$.
\end{enumerate}
Once these functions are selected, we define $S : c \to C(K)$
by setting $S \delta_i = g_i$ ($i \in \NN$), and $S \one = g$.
Then $\|(S-U)|_{c_0}\| \leq 4\sigma$, and $\|(S-U) \one\| \leq 3\sigma$,
hence, by Lemma \ref{l:norm c}, $\|S - U\| \leq 11\sigma$.

Moreover, $S$ is disjointness preserving. Indeed, consider two
disjoint elements of $c$: $x = (\alpha_i)_{i \in A}$ and
$y = (\beta_i)_{i \in B}$,
where the sets $A$ and $B$ are disjoint. If the sets
$\{i \in A : \alpha_i \neq 0\}$ and $\{i \in B : \beta_i \neq 0\}$
are both infinite, then
$x$ and $y$ belong to $c_0$, and we finish the proof invoking the disjointess
of the functions $g_i$. Otherwise, suppose $A$ is finite. Then we
can assume that $B = \NN \backslash A$. Let $\beta = \lim_i \beta_i$,
and write
$$
y = \beta \one + \sum_{i=1}^\infty \gamma_i \delta_i , \, \, {\textrm{where}} \,
\gamma_i = \left\{ \begin{array}{ll}
   \beta_i - \beta   &   i \in B   \\
   - 1               &   i \in A       \end{array} \right. .
$$
Then $Sx = \sum_{i \in A} \alpha_i g_i$, and
$$
Sy = g - \sum_{i \in A} g_i + \sum_{i \in B} \gamma_i g_i .
$$
If $[Sx](t) \neq 0$, then there exists $i \in A$ so that $g_i(t) \neq 0$,
and therefore, $[Sy](t) = g(t) - g_i(t) = 0$. Thus, $Sx$ and $Sy$ are disjoint.

Finally, suppose $g, g_1, g_2, \ldots$ are positive. For
$x = (\alpha_1, \alpha_2, \ldots) \in c_+$, let $\alpha = \lim_i \alpha_i$.
Then
$$
Sx = \alpha g + \sum_{i=1}^\infty (\alpha_i - \alpha) g_i \geq 0 .
$$
Indeed, suppose $t \in K$ is such that there exists $i$ with $g_i(t) > 0$.
Such an $i$ is unique, hence
$$
[Sx](t) = \alpha g(t) - (\alpha_i - \alpha) g_i(t) = \alpha_i g(t) \geq 0 .
$$
If there is no such $i$, then $[Sx](t) = \alpha g(t) \geq 0$.

To construct $g, g_1, g_2, \ldots$,
let $h = \phi_\sigma(f)$ (that is,
$h(t) = (|f(t)| - \sigma)_+ e^{\iota \arg f(t)}$).
For $i \in \NN$, set $h_i = \phi_\sigma (f_i)$.
Clearly $\|f-h\| \leq \sigma$, and $\|f_i - h_i\| \leq \sigma$ for any $i$.
Also, if $i$ and $t$ are such that $h_i(t) \neq 0$, then
$|h(t) - h_i(t)| \leq \sigma$, by Lemma \ref{l:complex} and \ref{eq:DP on c}.

Now define $\rho : \RR \to [0,1]$ via
$$
\rho(t) = \left\{ \begin{array}{ll}
   0          &     t \leq 0,     \\
   t/\sigma   &     0 \leq t \leq \sigma,     \\
   1          &     t \geq \sigma,
\end{array} \right.
$$
and let
$$
k_i(t) = \Big(1 - \rho \big(|h_i(t)| \big) \Big)  h_i(t) + \rho \big(|h_i(t)| \big) h(t) .
$$
Clearly the function $k_i$ is continuous, and $k_i(t) = 0$ whenever $h_i(t) = 0$.
If $h_i(t) \neq 0$, then
$$
\big|k_i(t) - h_i(t)\big| =
\rho \big(|h_i(t)| \big) \big|h(t) - h_i(t)\big| < \sigma ,
$$
hence $\|h_i - k_i\| \leq \sigma$. Finally, if $|k_i(t)| > 2 \sigma$, then
$k_i(t) = h(t)$. Indeed, if $|k_i(t)| > 2 \sigma$, then $|h_i(t)| > \sigma$,
hence $\rho \big(|h_i(t)| \big) = 1$, yielding $k_i(t) = h(t)$.

Now set $g_i = \phi_{2\sigma}(k_i)$, and $g = \phi_{2\sigma}(h)$.
From the above, if $g_i(t) \neq 0$, then $g_i(t) = g(t)$. Clearly the functions
$g_i$ are disjoint. Furthermore,
$$
\|f_i - g_i\| \leq \|f_i - h_i\| + \|h_i - k_i\| + \|k_i - g_i\| \leq 4 \sigma ,
$$
and
$$
\|f - g\| \leq \|f - h\| + \|h - g\| \leq 3 \sigma .
$$
Thus, $g, g_1, g_2, \ldots$ have the desired properties.
\end{proof}

\begin{corollary}\label{c:c disj}
Suppose $K$ is a compact Hausdorff space, and a contraction
$U : c \to C(K)$ is $\sigma$-DP.
Suppose, moreover, that $U|_{c_0}$ is disjointness preserving.
Then there exists a DP operator $S : c \to C(K)$ so that
$\|U - S\| \leq 11 \sigma$. If $U$ is positive, then
$S$ can be chosen positive as well.
\end{corollary}

\begin{proof}
Let $f_i = U \delta_i$ and $f = U \one$. The functions $f_i$ are
disjoint. Now fix $i$ and $t$, and set $x = \delta_i$ and $y = \one - \delta_i$.
Both $x$ and $y$ belong to the unit ball of $c$, hence
$$
|f_i(t)| \wedge |f(t) - f_i(t)| \leq \big\| |Tx| \wedge |Ty| \big\| \leq \sigma .
$$
Thus, \eqref{eq:DP on c} holds. To complete the proof, apply Lemma \ref{l:c disj}.
\end{proof}

\begin{proof}[Proof of Theorem \ref{t:c->C(K)}]
By Theorem \ref{t:c0->C(K)}, there exists a disjointness preserving
map $V : c_0 \to C(K)$ so that $\|V\| \leq \|T\|$, and
$\|V - T|_{c_0}\| \leq 257 \vr$ (if $T$ is positive, then $0 \leq V \leq T$).
Define $U : c \to C(K)$ by setting $U|_{c_0} = V$ and $U \one = T \one$.
By Lemma \ref{l:norm c}, $\|T - U\| \leq 514 \vr$.

Set $f = T \one = U \one$, $f_i = U \delta_i$, and $F_i = T \delta_i$.
Note that, if $T$ is positive, then so is $V$. Indeed, by the construction
in the proof of Theorem \ref{t:c0->C(K)}, $0 \leq f_i \leq F_i$ for every $i$.
Note that $T(\one - \delta_i) = f - F_i \geq 0$ for every $i$, hence
$f \geq f_i$. For $x = (\alpha_1, \alpha_2, \ldots) \in c_+$ set
$\alpha = \lim_i \alpha_i$, then
$$
Ux = \alpha f + \sum_{i=1}^n (\alpha_i - \alpha) f_i .
$$
Fix $t \in K$. If $f_i(t) = 0$ for every $i$, then $[Ux](t) = \alpha f(t) \geq 0$.
Otherwise, there is a unique $i$ so that $f_i(t) > 0$, then
$$
[Ux](t) = \alpha f(t) + (\alpha_i - \alpha) f_i(t) =
\alpha_i f_i(t) + \alpha(f(t) - f_i(t)) \geq 0 .
$$

We shall show that \eqref{eq:DP on c} holds with $\sigma = 2\vr$ --
that is, if $i$ and $t$ satisfy $f_i(t) \neq 0$, then
$|f_i(t) - f(t)| \leq 2\vr$. Once this is done, we can apply
the proof of Lemma \ref{l:c disj} to obtain $S$ with the desired properties.

Let $x = \delta_i$ and $y = \one - \delta_i$. In the above notation,
$Tx = F_i$ and $Ty = f - F_i$, hence, for any $t \in K$,
$\min\{ |F_i(t)|, |f(t) - F_i(t)| \} \leq \vr$.
By the proof of Theorem \ref{t:c0->C(K)},
$|F_i(t) - f_i(t)| \leq \vr$ (we use the fact that
$|\phi(s) - s| \leq \vr$ for every $s$).

Now suppose $|f_i(t)| \geq 2 \vr$. Then $|F_i(t)| \geq 2 \vr$ as well,
hence $|f(t) - F_i(t)| \leq \vr$. The triangle inequality implies
$$
|f(t) - f_i(t)| \leq |f(t) - F_i(t)| + |f_i(t) - F_i(t)| \leq 2\vr .
$$
By the proof of Lemma \ref{l:c disj}, there exists a ``good'' $S$
with $\|U - S\| \leq 22\vr$. By the triangle inequality, $\|T - S\| \leq 536 \vr$.
\end{proof}

\section{Positive operators from $\ell_p$ to $L_p$}\label{s:on l_1}

We start this section exploring the case of $\vr$-DP operators defined on the space $\ell_1$.
We use the following classical result of L. Dor \cite[Corollary 3.2]{Dor}.
Suppose $(\Omega,\mu)$ is a measure space, $(f_n)_{n\in\mathbb N}$
are functions in $L_1(\Omega, \mu)$, and there exists $\theta \in (0,1]$
so that the inequality $\|\sum_{i=1}^n a_i f_i\|\geq\theta\sum_{i=1}^n|a_i|$
holds for any finite sequence $(a_i)_{i=1}^n$.
Then there are disjoint measurable sets $(A_n)_{n\in\mathbb{N}}$ in $\Omega$ so that
$$
\inf_n \int_{A_n} |f_n|d\lambda\geq 1-\frac43(1-\theta).
$$
Dor proved this theorem for the Lebesgue measure on $[0,1]$. However
(as noted in e.g. \cite{Al}) an inspection shows that the proof works
for an arbitrary measure space. Moreover, one can select the sets $A_i$ from
the $\sigma$-algebra generated by the functions $(f_n)_{n\in\mathbb N}$.

\begin{theorem}\label{t:ell1}
Suppose $(\Omega,\mu)$ is a measure space, and
$T:\ell_1\rightarrow L_1(\mu)$ is a positive $\vr$-DP operator, with
$\vr \in (0,\|T\|/16)$.
Then there exists a positive disjointness preserving operator
$S:\ell_1\rightarrow L_1(\mu)$ such that $0\leq S\leq T$ and
$\|T-S\| \leq 2 \sqrt{2\vr\|T\|/3}$.
\end{theorem}

\begin{proof}
As usual, we can assume $\|T\| = 1$.
Then we need to prove the existence of a disjointess
preserving $S:\ell_1\rightarrow L_1(\mu)$ such that $0\leq S\leq T$ and
$\|T-S\| \leq 2 \sqrt{2\vr/3}$.

For $n\in \mathbb{N}$, let $f_n=T\delta_n$.
Since $\|T\|\leq1$ we have $\|f_n\|\leq1$.
By positivity, $f_n \geq 0$. Let
$$
c = 2 \sqrt{2\vr/3}  \, {\textrm{  and  }} \,
M=\big\{n\in\mathbb{N}:\|f_n\|\geq c\big\} .
$$

Now, for $n\in M$ let $g_n=f_n/\|f_n\|$.
These form a normalized sequence in $L_1(\mu)$
which is equivalent to the unit vector basis of $\ell_1$.
In fact, given real scalars $(a_n)_{n\in M}$, let
$P=\{n\in M : a_n>0\}$, $N=\{n\in M : a_n<0\}$ and
$x=\sum_{n\in P} |a_n| g_n$, $y=\sum_{n\in N} |a_n| g_n$. We have
\begin{align*}
\Big\|\sum_{n\in M} a_n g_n\Big\|
&=\Big\|\sum_{n\in P} |a_n| g_n-\sum_{n\in N} |a_n| g_n\Big\|\\
&=\Big\|x-x\wedge y+ x\wedge y- y\Big\|\\
&=\Big\|x-x\wedge y\Big\|+ \Big\|x\wedge y- y\Big\|\\
&\geq\|x\|+\|y\|-2\|x\wedge y\| .
\end{align*}

Since $g_n\geq0$ and $\|g_n\|=1$, we have $\|x\|=\sum_{n\in P}|a_n|$ and
$\|y\|=\sum_{n\in N} |a_n|$. Now, since $P\cap N=\emptyset$, and $P,N\subset M$ we have
\begin{align*}
\|x\wedge y\|
&=\Big\|\Big(\sum_{n\in P}  \frac{|a_n|}{\|f_n\|}f_n\Big)
  \wedge\Big(\sum_{n\in N} \frac{|a_n|}{\|f_n\|}f_n\Big)\Big\|\\
&= \Big\|T\Big(\sum_{n\in P}  \frac{|a_n|}{\|f_n\|}\delta_n\Big)
  \wedge T\Big(\sum_{n\in N} \frac{|a_n|}{\|f_n\|}\delta_n\Big)\Big\|\\
&\leq
 \vr \max\Big\{ \sum_{n\in P}  \frac{|a_n|}{\|f_n\|} ,
    \sum_{n\in N}  \frac{|a_n|}{\|f_n\|} \Big\}\\
&\leq \frac{\vr}{c} (\|x\| + \|y\|) .
\end{align*}

Hence, we get that
$$
\Big\|\sum_{n\in M} a_n g_n\Big\|\geq \Big(1-\frac{2\vr}{c}\Big)\sum_{n\in M}|a_n|.
$$

Now, by Dor's theorem quoted above, 
there exist pairwise disjoint measurable sets
$(A_n)\subset \Omega$ such that
$$
\|g_n|_{A_n}\|\geq 1-\frac{8 \vr}{3c} = 1-c .
$$

Let us define now the operator $S:\ell_1\rightarrow L_1(\mu)$ given by
$$
S\delta_n=\left\{
\begin{array}{ccc}
 f_n|_{A_n} &   & n\in \,M  \\
  &   &   \\
0  &   & \textrm{elsewhere}.
\end{array}
\right.
$$

Since the $(A_n)$ are pairwise disjoint, $S$ is disjointness preserving.
We have $\|T-S\| = \sup_n \|(T-S)\delta_n\|$.
Now, for $n\in M$ we have
$$
\|(T-S)\delta_n\|=\|f_n|_{A_n^c}\|=\|f_n\|-\|f_n|_{A_n}\|=
\|f_n\|(1-\|g_n|_{A_n}\|) \leq c,
$$
while for $n\notin M$ we get
$$
\|(T-S)\delta_n\|=\|f_n\|\leq c .
$$
Thus, $\|T-S\| \leq c$.
\end{proof}

\begin{theorem}\label{t:lq->Lq}
Suppose $1 < q < \infty$, $\vr \in (0,1/8^{\frac1q})$, and $(\Omega, \mu)$ is a measure space.
If $T : \ell_q \to L_q(\mu)$ is positive and $\vr$-DP,
then there exists $S : \ell_q \to L_q(\mu)$ so that $0 \leq S \leq T$,
and
$$
\|T-S\| \leq 2^8\vr + 2\sqrt\frac{2\vr\|T\|}{3} .
$$
\end{theorem}

To deduce this theorem from Theorem \ref{t:ell1}, we need an auxiliary result.

\begin{lemma}\label{l:norm}
Suppose $1 \leq q \leq \infty$, $(\Omega, \mu)$ is a measure space,
and the positive operator $R : \ell_q \to L_q(\mu)$ satisfies:
\begin{enumerate}
\item
If $x, y \in \ball(\ell_q)_+$ are disjoint, then $\| Rx \wedge Ry \| \leq \vr_1$.
\item
$\sup_i \|R \delta_i\| \leq \vr_2$, where $(\delta_i)_{i=1}^\infty$ is the canonical
basis of $\ell_q$.
\end{enumerate}
Then $\|R\| \leq 2^8 \vr_1 + \vr_2$.
\end{lemma}

\begin{proof}
Write $R \delta_i = f_i$, then $\sup_i \|f_i\| \leq \vr_2$. It suffices to show that
$\|\sum_{i=1}^n \alpha_i f_i\| \leq 2^8 \vr_1 + \vr_2$ whenever
$\alpha_1, \ldots, \alpha_n \geq 0$ satisfy $\sum_i \alpha_i^q \leq 1$.
By the triangle inequality,
\begin{equation}
\label{eq:tri Lq}
\|\sum_{i=1}^n \alpha_i f_i\| \leq
\|\sum_{i=1}^n \alpha_i f_i - \vee_{i=1}^n \alpha_i f_i\| + \|\vee_{i=1}^n \alpha_i f_i\| .
\end{equation}
However,
\begin{align*}
\big\|\vee_{i=1}^n \alpha_i f_i\big\|^q
&
\leq
\Big\| \big( \sum_{i=1}^n (\alpha_i f_i)^q \big)^{1/q} \Big\|^q
\\
&
=
\int \sum_{i=1}^n \alpha_i^q f_i(t)^q d \mu(t) \leq
\sup_{1\leq i\leq n} \|f_i\|^q \cdot \sum_{i=1}^n \alpha_i^q \leq \vr_2^q .
\end{align*}
Furthermore, by Corollary \ref{c:max Lq},
\begin{align*}
\|\sum_{i=1}^n \alpha_i f_i - \vee_{i=1}^n \alpha_i f_i\|
&
\leq
2^8 \expe_S \Big\| \big(\sum_{i \in S} \alpha_i f_i\big) \wedge
 \big(\sum_{i \in S^c} \alpha_i f_i\big) \Big\|
\\
&
=
2^8 \expe_S \Big\| R \big(\sum_{i \in S} \alpha_i \delta_i\big) \wedge
 R \big(\sum_{i \in S^c} \alpha_i \delta_i\big) \Big\| \leq 2^8 \vr_1
\end{align*}
(we average over all $S \subset \{1, \ldots, n\}$).
Plugging this into \eqref{eq:tri Lq}, we finish the proof.
\end{proof}

\begin{proof}[Proof of Theorem \ref{t:lq->Lq}]
By scaling, we can assume $\|T\| \leq 1$. We denote the canonical basis on $\ell_p$ by $(\delta_i^{[p]})_{i=1}^\infty$
(below, we consider $p = q$ and $p = 1$).
Let $f_i = T \delta_i^{[q]} \in L_q(\mu)$, and $g_i = f_i^q \in L_1$.
Define $T^\prime : \ell_1 \to L_1(\mu)$ by setting $T^\prime \delta_i^{[1]} = g_i$,
for every $i$. Clearly,
$$
\|T^\prime\| = \sup_i \|T \delta_i^{[1]}\|_1 = \sup_i \|g_i\|_1 =
\sup_i \|f_i\|_q^q = \sup_i \|T \delta_i^{[q]}\|_q^q \leq \|T\|^q \leq 1 .
$$
We show that $T^\prime$ is $\vr^q$-DP.
It suffices to prove that, for disjoint $x, y \in \ell_1$ with finite support,
we have $\| |T^\prime x| \wedge |T^\prime y| \|_1 \leq \vr^{q} \max\{\|x\|_1,\|y\|_1\}$.
Write $x = \sum_{i \in A} \alpha_i \delta_i^{[1]}$ and
$y = \sum_{i \in B} \beta_i \delta_i^{[1]} \in \ball(\ell_1)$,
where $A$ and $B$ are disjoint finite sets.
Define $\tilde{x} = \sum_{i \in A} |\alpha_i|^{1/q} \delta_i^{[q]},
\tilde{y} = \sum_{i \in B} |\beta_i|^{1/q} \delta_i^{[q]} \in \ell_q$.
Then
\begin{align*}
\Big\| |T^\prime x| \wedge |T^\prime y| \Big\|_1
&
\leq
\Big\| \big( \sum_{i \in A} |\alpha_i| g_i \big) \wedge
 \big( \sum_{i \in B} |\beta_i| g_i \big) \Big\|_1
\\
&
=
\int \big( \sum_{i \in A} |\alpha_i| g_i(t) \big) \wedge
 \big( \sum_{i \in B} |\beta_i| g_i(t) \big) d \mu(t) .
\end{align*}
However, it is easy to see that, for any positive $\gamma_1, \ldots, \gamma_m$,
we have $\sum_i \gamma_i \leq (\sum_i \gamma_i^{1/q})^q$, hence
\begin{align*}
\Big\| |T^\prime x| \wedge |T^\prime y| \Big\|_1
&
\leq
\int \Big( \big( \sum_{i \in A} |\alpha_i|^{1/q} f_i(t) \big) \wedge
 \big( \sum_{i \in B} |\beta_i|^{1/q} f_i(t) \big) \Big)^q d \mu(t)
\\
&
=
\big\| (T \tilde{x}) \wedge (T \tilde{y}) \big\|_q^q \leq
\vr^q \max\{\|\tilde{x}\|_q^{q}, \|\tilde{y}\|_q^{q}\} =
\vr^q  \max\{\|x\|_1,\|y\|_1\}.
\end{align*}
Use Theorem \ref{t:ell1} to find an operator $S^\prime : \ell_1 \to L_1(\mu)$ so that
$0 \leq S^\prime \leq T^\prime$, and $\|T^\prime - S^\prime\| \leq (8/3)^{1/2} \vr^{q/2}$.
Define $S : \ell_q \to L_q$ by setting
$$
S(\sum_i \alpha_i \delta_i^{[q]}) = \sum_i \alpha_i (S^\prime \delta_i^{[1]})^{1/q}.
$$
We clearly have $0 \leq S \leq T$, hence $S$ is a bounded operator.
It remains to estimate $\|T - S\|$ from above.

As $0 \leq T-S \leq T$, $T-S$ must be $\vr$-DP. Furthermore,
for any $i$,
$$
\|(T-S)\delta_i^{[q]}\|_q^q =
\|T\delta_i^{[q]} - S\delta_i^{[q]}\|_q^q =
\int \Big( (T\delta_i^{[q]})(t) - (S\delta_i^{[q]})(t) \Big)^q d \mu(t) .
$$
Note that, for $0 \leq \alpha \leq \beta$, we have
$(\beta - \alpha)^q \leq \beta^q - \alpha^q$.
Recall that $(T\delta_i^{[q]})(t) = f_i(t) = g_i(t)^{1/q} = (T^\prime \delta_i^{[1]})(t)^{1/q}$,
and $(S\delta_i^{[q]})(t) = (S^\prime \delta_i^{[1]})(t)^{1/q}$. Thus,
$$
\|(T-S)\delta_i^{[q]}\|_q^q \leq
\int \Big( (T^\prime\delta_i^{[1]})(t) - (S^\prime\delta_i^{[1]})(t) \Big) d \mu(t)
\leq \|T^\prime - S^\prime\| \leq \sqrt{\frac{8}{3}} \vr^{q/2} .
$$
Lemma \ref{l:norm} gives the desired estimate for $\|T - S\|$.
\end{proof}

\begin{remark}\label{r:almost isometries}
It is well-known that for $p\neq2$ every linear isometry
$T:L_p(\mu)\rightarrow L_p(\nu)$ is disjointness preserving (cf. \cite[p. 77]{Car}).
Along the same lines, it can be shown that for $p\neq 2$, there is a constant
$C_p$ such that every linear $\vr$-isometry $T:L_p(\mu)\rightarrow L_p(\nu)$ (that is, $(1+\vr)^{-1}\|x\|\leq\|Tx\|\leq(1+\vr)\|x\|$), is also $C_p\vr$-DP.
\end{remark}

\section{Positive operators from sequence spaces to $L_1$}\label{s:to L1}

Throughout this section, the Banach lattice structure on $E$ is assumed to be
given by its $1$-unconditional basis $(\delta_i)$.

Denote by $\sph(Z)$ the unit sphere of a normed space $Z$.
We define the set-valued \emph{duality mapping} $\dual$
by letting, for $x \in E \backslash \{0\}$,
$\dual(x) = \{f \in \sph(E^*) : f(x) = \|x\|\}$.
The map $\dual$ is said to be \emph{lower semicontinuous} if, for any
$x \in E \backslash \{0\}$, and any open set $U$ with $U \cap \dual(x) \neq \emptyset$,
there exists $\vr \in (0,\|x\|)$ so that $U \cap \dual(y) \neq \emptyset$
whenever $\|x-y\| < \vr$.

We call the space $E$ \emph{smooth} if $\dual(x)$ is a singleton for very $x$.
In this case, we can define $\ovdual : E \backslash \{0\} \to E^*$ so that
$\dual(x) = \{\ovdual(x)\}$ for every $x$. It is known (see \cite[Section 2.2]{Die75})
that $\ovdual$ is continuous (with respect to the norm topology) if and only if the norm
of $E$ is Fr\'echet differentiable away from $0$. Clearly, for smooth spaces
$\ovdual$ is continuous if and only if $\dual$ is lower semi-continuous.

\begin{theorem}\label{t:FD->l1}
Suppose the order in a reflexive Banach lattice $E$ is determined by its
$1$-unconditional basis, and the duality map is lower semi-continuous on $E \backslash \{0\}$.
Suppose, furthermore, that the operator $T \in B(E,\ell_1)_+$ is $\vr$-DP. Then there exists a disjointness preserving operator $S \in B(E,\ell_1)_+$ so that
$S \leq T$, and $\|T - S\| \leq 256 \vr$.
\end{theorem}

Let us begin with some auxiliary results. The first one is straightforward.

\begin{lemma}\label{l:norm into L_1}
If $E$ is a space with a $1$-unconditional basis $\delta_i$, then, for any
$T \in B(E,L_1(\mu))_+$,
$$
\|T\| = \Big\|\big( \|T \delta_i\| \big)_i \Big\|_{E^*}.
$$
\end{lemma}

\begin{proof}
For the sake of brevity, set $f_i = T \delta_i$.
Suppose $(\alpha_i) \in c_{00}$ is a finite sequence of non-negative numbers,
then
$$
\big\| T \big(\sum_i \alpha_i \delta_i\big) \big\| = \int \big( \sum_i \alpha_i f_i \big) =
\sum_i \alpha_i \|f_i\| .
$$
Therefore,
\begin{align*}
\|T\|
&
=
\sup \Big\{ \big\| T \big(\sum_i \alpha_i \delta_i\big) \big\| :
 \big\| \sum_i \alpha_i \delta_i \big\| \leq 1 \Big\}
\\
&
=
 \sup \Big\{ \sum_i \alpha_i \|f_i\| :
 \big\| \sum_i \alpha_i \delta_i \big\| \leq 1 \Big\} =
 \big\| ( \|f_i\| ) \big\|_{E^*} .
\end{align*}
\end{proof}

The next lemma may be known to the experts in Banach space geometry.

\begin{lemma}\label{l:support}
Suppose $Z$ is a real Banach space whose duality mapping $\dual$ is lower
semi-continuous. Suppose, furthermore, that there exist
$z, z_1, z_2, \ldots \in Z$ so that $z \neq 0$, $\lim_n \|z - z_n\| = 0$,
and for each $n$ there exists $z^*_n \in \dual(z)$ so that
$$
\limsup_n \frac{\|z\| - \langle z^*_n, z_n \rangle}{\|z - z_n\|} > 0 .
$$
Then $\|z_n\| < \|z\|$ for some value of $n$.
\end{lemma}

\begin{proof}
By rescaling, we can assume that $\|z\| = 1$. Furthermore, by passing to a subsequence,
we can assume that, for every $n$,
$$
\langle z_n^*, z_n \rangle < 1 - c\|z - z_n\|,
$$
where $c > 0$ is a constant.
By the lower semi-continuity of the duality map, we can find
a sequence $\tilde{z}_n^* \in \dual(z_n)$ so that
$\lim_n \|z^*_n - \tilde{z}_n^*\| = 0$. We then have
\begin{equation}
\label{eq:norming}
\|z_n\| = \langle \tilde{z}_n^* , z_n \rangle =
\langle \tilde{z}_n^* , z \rangle - \langle z_n^* , z \rangle +
\langle \tilde{z}_n^* - z_n^* , z_n - z \rangle +
\langle z_n^*, z_n \rangle .
\end{equation}
As $z_n^* \in \dual(z)$, and $\|\tilde{z}_n^*\| = 1$, we have
$\langle \tilde{z}_n^* , z \rangle - \langle z_n^* , z \rangle \leq 0$.
Furthermore, $\langle z_n^*, z_n \rangle \leq 1 - c\|z - z_n\|$, and
$$
\langle \tilde{z}_n^* - z_n^* , z_n - z \rangle \leq
\|\tilde{z}_n^* - z_n^*\| \|z_n - z\| = o \Big(\|z - z_n\| \Big) .
$$
Now \eqref{eq:norming} shows that $\|z_n\| \leq 1 - c\|z - z_n\| + o(\|z - z_n\|)$.
\end{proof}

\begin{proof}[Proof of Theorem \ref{t:FD->l1}]
We can and do assume that the basis $(\delta_i)$ is normalized. Let $f_i=T\delta_i$.
By Corollary \ref{c:arb number}, for every sequence $(\alpha_i) \in c_{00}$
we have
$$
\|\sum_i \alpha_i f_i - \vee_i \alpha_i f_i\| \leq
 256 \vr \|\sum_i \alpha_i \delta_i\|.
$$

We will find mutually disjoint sets $A_i \subset \NN$ with the property that
\begin{equation}
\label{eq:remainder}
\big\| \sum_i \|\one_{A_i^c} f_i\| \delta_i^*\| \leq 256 \vr .
\end{equation}
Once this is done, we define $S : E \to \ell_1 : \delta_i \mapsto \one_{A_i} f_i$.
Then clearly $0 \leq S \leq T$, and by Lemma \ref{l:norm into L_1},
$$
\|T - S\| = \big\| \sum_i \|f_i -  \one_{A_i} f_i\| \delta_i^*\| =
 \big\| \sum_i \|\one_{A_i^c} f_i\| \delta_i^*\| \leq 256 \vr .
$$

For the purpose of finding $(A_i)$, we use some ideas of \cite{Dor}.
Consider the space
$$
{\mathcal{H}} = \big\{ (h_1, h_2, \ldots) \in \prod_i \ball(\ell_\infty)_+ :
 \sum_i h_i \leq \one \big\} .
$$
Here, $\prod_i \ball(\ell_\infty)_+$ is equipped with the topology
of the product of infinitely many copies of $(\ell_\infty, w^*)$.
It is easy to see that ${\mathcal{H}}$ is compact. Now define
$$
F : {\mathcal{H}} \to \RR : (h_i)_{i \in \NN} \mapsto
 \Big\| \sum_i \|(\one - h_i) f_i\| \delta_i^*\Big\| .
$$
Note that the function $F$ is convex. Indeed, suppose
$h_i = t h_i^{(0)} + (1-t) h_i^{(1)}$ for every $i$. For convenience, set
$\phi_i = f_i (\one - h_i)$, and $\phi_i^{(j)} = f_i (\one - h_i^{(j)})$
for $j = 0,1$. Then $\phi_i = t \phi_i^{(0)} + (1-t) \phi_i^{(1)}$,
and as all the functions are non-negative,
$\|\phi_i\| = t \|\phi_i^{(0)}\| + (1-t) \|\phi_i^{(1)}\|$.
\begin{align*}
F((h_i)_i)
&
=
\big\| \sum_i \|\phi_i\| \delta_i^*\| =
\Big\| \sum_i \big( t \|\phi_i^{(0)}\| + (1-t) \|\phi_i^{(1)}\| \big) \delta_i^* \Big\|
\\
&
\leq
t \big\| \sum_i \|\phi_i^{(0)}\| \delta_i^* \big\| +
 (1-t) \big\| \sum_i \|\phi_i^{(1)}\| \delta_i^* \big\|
\\
&
=
t F((h_i^{(0)})_i) + (1-t) F((h_i^{(0)})_i) .
\end{align*}
Moreover, $F$ is continuous. Indeed, fix $\vr' > 0$ and $(h_i) \in {\mathcal{H}}$.
Find $N$ so that $\|\sum_{i=N+1}^\infty \|f_i\| \delta_i^*\| < \vr'/2$.
Then $|F((h_i)) - F(h_i^\prime))| < \vr'$ whenever, for $1 \leq i \leq N$,
$$
\Big| \|(1-h_i)f_i\| -  \|(1-h_i^\prime)f_i\| \Big| =
\Big\|  (h_i^\prime - h_i) f_i \Big\| =
\big| \langle h_i - h_i^\prime , f_i \rangle \big| < \frac{\vr'}{2N}
$$
($\langle \cdot , \cdot \rangle$ denotes the duality bracket between
$\ell_\infty$ and $\ell_1$).
The centered equation above clearly defines a relatively open subset of ${\mathcal{H}}$.

By the above, for any $n \in \NN$ there exists an extreme point
$(h_i^{(n)})_i \in {\mathcal{H}}$ so that $F((h_i^{(n)})_i) < \inf F + 1/n$.
As noted in \cite{Dor}, $(h_i)$ is an extreme point of ${\mathcal{H}}$ if and only if there
exist disjoint sets $A_i$ so that $h_i = \one_{A_i}$, for every $i$.
Moreover, the set of the extreme points of ${\mathcal{H}}$ is closed.
Indeed, one can observe that ${\mathcal{H}}$ is metrizable.
Suppose $((h_i^{(n)})_{i \in \NN})_{n \in \NN}$ is a sequence of extreme points,
converging to some $(h_i)_{i \in \NN} \in {\mathcal{H}}$.
Write $h_i^{(n)} = \one_{A_i^{(n)}}$. Then for any $i$,
$h_i^{(n)} \underset{n}\rightarrow h_i$ pointwise, hence $h_i = \one_{A_i}$.
Moreover, for each $i, t \in \NN$, only two situations are possible:
\begin{enumerate}
\item
For $n$ large enough, $t \in A_i^{(n)}$ (that is, $h_i^{(n)}(t) = 1$),
and consequently, $t \in A_i$.
\item
For $n$ large enough, $t \notin A_i^{(n)}$, and then, $t \notin A_i$.
\end{enumerate}
This shows that the sets $(A_i)$ are disjoint.

We therefore conclude that $F$ attains its minimum on an extreme point
$(\one_{A_i})$. By enlarging the sets $A_i$ if necessary, we can assume
that $\cup_i A_i = \NN$. It remains to show that these sets satisfy
\eqref{eq:remainder}.

For the sake of brevity write $\beta_i = \|\one_{A_i^c} f_i\|$, and
$x = \sum_i \beta_i \delta_i^*$. Find
$z = \sum_i \alpha_i \delta_i \in \sph(E)_+$
so that $\sum_i \alpha_i \beta_i = \big\| \sum_i \beta_i \delta_i^*\|$.
We show that, for any $t \in A_i$, $\alpha_i f_i(t) = \vee_j \alpha_j f_j(t)$.
Indeed, suppose, for the sake of contradiction, that there exist $t \in A_i$,
and $j \neq i$, so that $\alpha_i f_i(t) < \alpha_j f_j(t)$.
For $k \in \NN$, let $h_k = \one_{A_k}$. Furthermore, for any
$\vr \in (0, (\alpha_j f_j(t) - \alpha_i f_i(t)))/2$, define
$h_k^{(\vr)}$ by setting $h_k^{(\vr)} = h_k$ for $k \notin \{i,j\}$,
$h_i^{(\vr)} = h_i - \vr \one_{\{t\}}$, and $h_j^{(\vr)} = h_j + \vr \one_{\{t\}}$
Let $\beta_k^{(\vr)} = \|(\one - h_k^{(\vr)}) f_k\|$, then
$\beta_k^{(\vr)} = \beta_k$ for $k \notin \{i,j\}$,
$\beta_i^{(\vr)} = \beta_i + \vr f_i(t)$, and $\beta_j^{(\vr)} = \beta_j - \vr f_j(t)$.
Write $x = \sum_k \beta_k \delta_k^*$, and $x^{(\vr)} = \sum_k \beta_k^{(\vr)} \delta_k^*$.
Then
$$
\big\| x - x^{(\vr)} \big\| = \big\| \vr f_i(t) \delta_i - \vr f_j(t) \delta_j \big\|
\leq \big( |f_i(t)| + |f_j(t)| \big) \vr .
$$
Moreover,
\begin{align*}
\langle z , x^{(\vr)} \rangle
&
=
\sum_k \alpha_k \beta_k^{(\vr)} =
\sum_k \alpha_k \beta_k + \vr \big(\alpha_i f_i(t) - \alpha_j f_j(t)\big)
\\
&
=
1 - \vr \big(\alpha_j f_j(t) - \alpha_i f_i(t)\big) .
\end{align*}
An application of Lemma \ref{l:support} shows that, for some $\vr$,
$$
F\big( (h_i^{(\vr)})_i \big) = \|x^{(\vr)}\| < \|x\| = F\big( (h_i)_i \big) ,
$$
contradicting our assumption that $F$ attains its minimum at $(h_i)$.

For $N \in \NN$, let $B_N = \cup_{k=1}^N A_k$ and
$\phi_N = \sum_{i=1}^N \alpha_i \one_{A_i^c}f_i$. By the above,
$\phi_N(t) = \sum_{i=1}^N \alpha_i f_i(t) - \vee_i \alpha_i f_i(t)$ for $t \in B_N$.
Consequently,
$$
\Big\| \phi_N \one_{B_N} \Big\| \leq
\Big\| \sum_{i=1}^N \alpha_i f_i - \vee_i \alpha_i f_i \Big\| \leq 256 \vr .
$$
Now consider a finite set $B \subset \NN$. Then $B \subset B_N$ for $N$
large enough, hence
$$
\Big\| \Big( \sum_{i=1}^N \alpha_i \one_{A_i^c}f_i \Big) \one_B \Big\| \leq 256 \vr
$$
for every $N$. By the Fatou Property of $\ell_1$,
$$
\Big\| \Big( \sum_{i=1}^\infty \alpha_i \one_{A_i^c}f_i \Big) \one_B \Big\| \leq 256 \vr ,
$$
and as $B$ can be arbitrarily large,
$\big\| \sum_i \alpha_i \one_{A_i^c}f_i \big\| \leq 256 \vr$. Now, since
$$
\big\| \sum_i \|\one_{A_i^c} f_i\| \delta_i^*\| =\sum_i\alpha_i \|\one_{A_i^c} f_i\| =\big\| \sum_i \alpha_i \one_{A_i^c}f_i \big\| \leq 256 \vr
$$
we get \eqref{eq:remainder} as claimed.
\end{proof}

\begin{theorem}\label{t:seq->l1}
Suppose the order in a reflexive Banach lattice $E$ is determined
by its $1$-unconditional basis, and the operator $T \in B(E,\ell_1)_+$ is $\vr$-DP. Then for every $c > 1$ there exists a disjointness preserving operator $S \in B(E,\ell_1)_+$ so that $S \leq T$, and $\|T - S\| \leq 256 c \vr$.
\end{theorem}

For the proof we need a renorming result similar to \cite[Proposition 1.4]{God}.
Recall that a Banach space $Z$ is called \emph{locally uniformly rotund}
(\emph{LUR} for short) if, for any $z, z_1, z_2, \ldots \in Z$,
$\lim \|z_n - z\| = 0$ whenever
$\lim_n \big(2(\|z\|^2 + \|z_n\|^2) - \|z+z_n\|^2\big) = 0$.
We say that that a basis in a Banach space $Z$ is \emph{shrinking}
if its biorthogonal functionals form a basis of the dual space $Z^*$. For unconditional bases this condition holds precisely when the space contains no subspace isomorphic to $\ell_1$ (\cite[Theorem 1.c.9]{LT1}.)

\begin{lemma}\label{l:FD renorm}
Suppose $(E, \| \cdot\|)$ is a space with a shrinking $1$-unconditional basis
$(\delta_i)$. Then for every $c > 1$, $E$ admits an equivalent norm
$\| \cdot \|_0$ such that:
\begin{enumerate}
\item
For any $x \in E$, $\|x\|_0 \leq \|x\| \leq c \|x\|_0$.
\item
$(E,\| \cdot \|_0)^*$ is LUR.
\item
The basis $(\delta_i)$ is $1$-unconditional in $(E,\| \cdot \|_0)$.
\end{enumerate}
\end{lemma}

\begin{proof}[Sketch of the proof]
We follow the reasoning of \cite[Proposition 1.4]{God}.
The minor changes that are required are indicated below.
As before, we assume that the basis $(\delta_i)$ is normalized,
and denote the cooresponding biorthogonal functionals by $\delta_i^*$.
To distinguish between the (originally given) norms on $E$ and $E^*$,
we denote them by $\| \cdot \|$ and $\| \cdot \|^*$, respectively.

Find $1 = \vr_0 > \vr_1 > \vr_2 > \ldots > 0$ so that
$\sum_{i=0}^\infty \vr_i < c$. For $f = \sum_i f_i \delta_i^* \in E^*$, set
$$
\|f\|^*_1 = \big( \|f\|^{*2} + \sum_i \vr_i |f_i|^2 \big)^{1/2} .
$$
Then $(E^*, \| \cdot \|_1^*)$ is smooth, and for any $f$,
$\|f\|^* \leq \|f\|^*_1  \leq \sqrt c \|f\|^*$. Moreover, $\| \cdot \|^*_1$
is a dual norm, and we can define the predual norm $\| \cdot \|_1$ on $E$.
Finally, the basis $(\delta_i^*)$ is $1$-unconditional in $(E^*, \| \cdot \|^*_1)$,
hence $(\delta_i)$ is $1$-unconditional in $E, \| \cdot \|_1)$.

Now set
$$
\|f\|^*_0 = \Big( \sum_{i=0}^\infty
 \vr_i \big\| \sum_{k=i+1}^\infty f_k \delta_k^* \big\|_1^{*2} \Big)^{1/2} .
$$
This is a dual LUR norm, and $\|f\|_1^* \leq \|f\|^*_0  \leq \sqrt c \|f\|^*_1$.
Finally, the $1$-unconditionality is once again preserved.
\end{proof}

\begin{proof}[Proof of Theorem \ref{t:seq->l1}]
By Lemma \ref{l:FD renorm}, we can equip $E$ with an equivalent norm $\| \cdot \|_0$,
with the properties that $\| \cdot \|_0 \leq \| \cdot \| \leq c \| \cdot \|_0$,
the basis $(\delta_i)_{i=1}^\infty$ is $1$-unconditional,
and $(E, \| \cdot \|_0)^*$ is LUR. By \cite[Corollary 1.16]{God},
$\| \cdot \|_0$ is Fr\'echet differentiable on $E \backslash \{0\}$.

Now consider $T$ as a map from $(E, \| \cdot \|_0)$ into $\ell_1$.
As $\ball(E, \| \cdot \|_0) \subset c \ball(E)$, we conclude that
$T$ is $c \vr$-DP with respect to $\|\cdot \|_0$.
By Theorem \ref{t:FD->l1}, we can find a disjointness preserving mapping
$S : (E, \| \cdot \|_0) \to \ell_1$ so that $0 \leq S \leq T$, and
$\|T - S\| \leq 256 c \vr$. To finish the proof, recall that
$\| \cdot \|_0 \leq \| \cdot \|$.
\end{proof}

In the case of operators with values in $L_1(\Omega,\mu)$
(for an arbitrary measure space $(\Omega,\mu)$), we obtain:

\begin{theorem}\label{t:seq->L1}
Suppose the order in a Banach lattice $E$ is determined
by its $1$-unconditional shrinking
basis, and the operator $T \in B(E,L_1(\Omega,\mu))_+$ is $\vr$-DP. Then for every $\sigma > 0$ there exists a disjointness preserving finite rank operator $S \in B(E,L_1(\Omega,\mu))_+$ so that $\|T - S\| \leq 256 \vr + \sigma$.
\end{theorem}

\begin{remark}
Note that every positive operator from a space with a shrinking unconditional basis into $L_1(\Omega,\mu)$ is necessarily compact.
\end{remark}

\begin{proof}
As before denote the normalized $1$-unconditional basis of $E$ by $(\delta_i)$,
and set $f_i = T \delta_i$. Then $E^*$ is spanned by $(\delta_i^*)_{i \in \NN}$,
and, by Lemma \ref{l:norm into L_1},
$\|T\| = \|\sum_{i=1}^\infty \|f_i\| \delta_i^*\|$. Given $\sigma>0$, find $N$ so that
$$
\|\sum_{i=N+1}^\infty \|f_i\| \delta_i^*\| < \sigma/4.
$$
Let $E_N = \spn[\delta_1, \ldots, \delta_N] \subset E$.
Find a finite $\sigma$-algebra ${\mathcal{A}}$ in $(\Omega,\mu)$,
so that, for every $x \in \ball(E_N)$,
$$
\|Tx - PTx\| < 2^{-11} \sigma
$$
(here $P$ denotes the conditional expectation onto $L_1({\mathcal{A}}, \mu)$).
Then $T^\prime = PT|_{E_N}$ is $(\vr + 2^{-10} \sigma)$-DP.
Indeed, for every disjoint $x_1, x_2 \in \ball(E_N)$,
\begin{align*}
&
\big\| |T^\prime x_1| \wedge |T^\prime x_2| \big\| \leq
\Big\| \big|T^\prime x_1\big| \wedge \big|(T^\prime - T) x_2 \big| \Big\| +
 \Big\| \big|T^\prime x_1\big| \wedge \big|T x_2\big| \Big\|
\\
&
\leq
\big\|(T^\prime - T) x_2 \big\| +
 \Big\| \big|(T^\prime - T) x_1\big| \wedge \big|Tx_2\big| \Big\| +
 \Big\| \big|T x_1\big| \wedge \big|Tx_2\big| \Big\|
\\
&
\leq
\big\|(T^\prime - T) x_2 \big\| + \big\|(T^\prime - T) x_1 \big\| +
 \Big\| \big|T x_1\big| \wedge \big|Tx_2\big| \Big\| \leq 2^{-10} \sigma + \vr.
\end{align*}

Fix $c \in (1, (256 \vr + \sigma/4)^{-1} (256 \vr + 3\sigma/4))$.
As in the proof of Theorem \ref{t:seq->l1}, we can find
$S^\prime : E_N \to L_1({\mathcal{A}},\mu)$ so that
$0 \leq S^\prime \leq T^\prime$, and
$\|S - T\| \leq (256 \vr + \sigma/4) c$.
Now define $S : E \to L_1(\Omega,\mu)$ by setting $S \delta_i = S^\prime \delta_i$
for $1 \leq i \leq N$, $S \delta_i = 0$ otherwise. Clearly $S$ is
positive and disjointness preserving, and
$$
\big\|T-S\big\| \leq \big\|T^\prime - S^\prime\big\| +
 \big\|T|_{\spn[\delta_i : i > N]}\big\| \leq
256 c \big(2^{-10} \sigma + \vr\big) + \frac{\sigma}{4} < 256 \vr + \sigma ,
$$
due to the choice of $c$.
\end{proof}

\section{Counterexamples}\label{s:counter}

In this section we show that, in general, not every positive almost DP operator
can be approximated by a disjointness preserving one. Actually, our examples
produce positive operators $T$ which are not merely $\vr$-DP, but
have a stronger property: $\||Tx|\wedge|Ty|\|\leq\vr\sqrt{\|x\|\|y\|}$ for any $x\perp y$.

\begin{proposition}\label{p:main}
Suppose $1 \leq p < q < \infty$. Then for any $\vr > 0$ there exists a finite rank
positive $\vr-\DP$ operator $T : \ell_p \to \ell_q$, so that $\|T\| \leq 2^{1-1/q}$,
and $\|T - S\| \geq 2^{-1/q} \geq \|T\|/2$ whenever $S$ is disjointness preserving.
\end{proposition}

Start with a combinatorial lemma.

\begin{lemma}\label{l:graph}
For $N \in \NN$, let $M = N(N+1)/2$. Then $\{1, \ldots, M\}$ contains sets
$F_1, \ldots, F_{N+1}$ of cardinality $N$ each, so that (i) each number
$s \in \{1, \ldots, M\}$ belongs to exactly two of the sets $F_i$;
(ii) $|F_i \cap F_j| = 1$ if $i \neq j$.
\end{lemma}

\begin{proof}
Consider the complete graph on $N+1$ vertices, and denote its sets of vertices
and edges by $V$ and $E$ respectively. Write $V = \{v_1, \ldots, v_{N+1}\}$
and $E = \{e_1, \ldots, e_M\}$. Let $F_i$ be the set of all $s$ so that
$e_s$ is adjacent to $v_i$.
\end{proof}

\begin{proof}[Proof of Proposition \ref{p:main}]
Pick $N \in \NN$ so that
$$
\vr \geq \left\{ \begin{array}{ll}
   N^{-1/q}   &   \infty > q \geq 2p,   \\
   \Big( N^{-1} (N+1)^{2 - q/p} \Big)^{1/q}   &   2p > q > p.
\end{array} \right.
$$
Define the operator $T : \ell_p^{N+1} \to \ell_q^M$ by setting
$T \delta_i = N^{-1/q} \one_{F_i}$, where $(\delta_i)$ is the canonical basis
for $\ell_p^{N+1}$. Clearly, $T$ is positive. Moreover,
$$
\|T : \ell_1^{N+1} \to \ell_1^M\| = \max_i \|T \delta_i\|_1 = N^{1/q^\prime} ,
$$
where $1/q + 1/q^\prime = 1$. Furthermore,
$$
\|T : \ell_\infty^{N+1} \to \ell_\infty^M\| = \|T \one\|_\infty =
 N^{-1/q} \|\sum_i \one_{F_i}\| = 2 N^{-1/q}
$$
(for $1 \leq s \leq M$, $\big(\sum_i \one_{F_i}\big)(s) = 2$,
since $s \in F_i$ for exactly two indices $i$). By interpolation,
$$
\|T : \ell_q^{N+1} \to \ell_q^M\| \leq
\|T : \ell_1^{N+1} \to \ell_1^M\|^{1/q}
 \|T : \ell_\infty^{N+1} \to \ell_\infty^M\|^{1/q^\prime} \leq
2^{1/q^\prime} .
$$
As the formal identity from $\ell_p^{N+1}$ to $\ell_q^{N+1}$ is contractive,
the desired estimate for $\|T\|$ follows.

Next show that $T$ is $\vr-\DP$. Consider disjoint elements
\begin{align*}
&
x = \sum_{i \in P_x} \alpha_i \delta_i  {\textrm{   and   }}
y = \sum_{j \in P_y} \beta_j \delta_j ,
\\
&
{\textrm{   where   }}
P_x \cap P_y = \emptyset  {\textrm{  and  }}
P_x \cup P_y = \{1, \ldots, N+1\} .
\end{align*}
For $s \in \{1, \ldots, M\}$ let $Q_s$ be the set of $i$'s for which
$s \in F_i$ (we have $|Q_s| = 2$). If $Q_s \subset P_x$ or $Q_s \subset P_y$,
then $(|Tx| \wedge |Ty|)(s) = 0$. If $Q_s = \{i,j\}$ with $i \in P_x$ and
$j \in P_y$, then
$$
N^{1/q} \big(|Tx| \wedge |Ty|\big)(s) = |\alpha_i| \wedge |\beta_j| \leq
|\alpha_i|^{1/2} |\beta_j|^{1/2} .
$$
Note that any pair $(i,j)$ appears in the right hand side of the centered
inequality at most once (when $Q_s = (i,j)$). Therefore,
$$   \eqalign{
N \big\||Tx| \wedge |Ty|\big\|_q^q
&
=
N \sum_s \big|\big(|Tx| \wedge |Ty|\big)(s)\big|^q \leq
\sum_{i,j} \big( |\alpha_i|^{1/2} |\beta_j|^{1/2} \big)^q
\cr
&
=
\sum_i |\alpha_i|^{q/2} \sum_j |\beta_j|^{q/2} .
}  $$

For $q \geq 2p$,
$$
\Big( \sum_i |\alpha_i|^{q/2} \Big)^{2/q} \leq
\Big( \sum_i |\alpha_i|^p \Big)^{1/p} = \|x\|_p ,
$$
and therefore, $\sum_i |\alpha_i|^{q/2} \leq \|x\|_p^{q/2}$.
Similarly, $\sum_j |\beta_j|^{q/2} \leq \|y\|_p^{q/2}$. Thus,
$$
\big\||Tx| \wedge |Ty|\big\|_q^2 \leq N^{-2/q} \|x\|_p \|y\|_p \leq
\vr^2 \|x\|_p \|y\|_p ,
$$
due to our definition of $\vr$.

For $p < q < 2p$,
$$
\Big( \sum_i |\alpha_i|^{q/2} \Big)^{2/q} \leq
(N+1)^{2/q - 1/p} \Big( \sum_i |\alpha_i|^p \Big)^{1/p} =
(N+1)^{2/q - 1/p} \|x\|_p ,
$$
hence
$$
\sum_i |\alpha_i|^{q/2} \leq (N+1)^{1 - q/(2p)} \|x\|_p^{q/2} .
$$
Handling $\sum_j |\beta_j|^{q/2}$ similarly, we conclude that
$$
N \big\||Tx| \wedge |Ty|\big\|_q^q \leq
 (N+1)^{2 - q/p} \|x\|_p^{q/2} \|y\|_p^{q/2} ,
$$
hence
$$
\big\||Tx| \wedge |Ty|\big\|_q \leq
 \Big( N^{-1} (N+1)^{2 - q/p} \Big)^{1/q} \sqrt{\|x\|_p \|y\|_p} \leq
 \vr \sqrt{\|x\|_p \|y\|_p} .
$$

Finally, we show that $T$ is poorly approximated by disjointness preserving
operators. Suppose $S : \ell_p^{N+1} \to \ell_q^M$ is disjointness preserving.
Let $G_i = \supp (S \delta_i)$ and $H_i = F_i \backslash G_i$. The sets $G_i$ are
disjoint, and $\sum_{i=1}^{N+1} |G_i| \leq M = N(N+1)/2$, hence
$|G_i| \leq N/2$ for some $i$. Then $|H_i| \geq N/2$, hence
$$
\|T-S\| \geq \|(T-S)\delta_i\| \geq N^{-1/q} |H_i|^{1/q} \geq 2^{-1/q} .
$$

Thus, $T$ has all the desired properties.
\end{proof}

The above results can be generalized somewhat (by extending the range space).
Recall that a Banach lattice $X$ \emph{satisfies a lower $q$-estimate
with constant $\cq$} if, for any disjoint $x_1, \ldots, x_n \in X$,
$\| \sum_i x_i \| \geq \cq \big( \sum_i \|x_i\|^q \big)^{1/q}$.

\begin{proposition}\label{p:lower q}
Suppose $1 \leq p < q < \infty$, and $X$ is an infinite dimensional Banach lattice,
satisfying a lower $q$-estimate with constant $\cq$.
Suppose, moreover, that $X$ does not satisfy a lower $r$-estimate for any $r < q$.
Then for any $\vr > 0$ there exists a finite rank
positive $\vr-\DP$ operator $T : \ell_p \to X$, so that
$\|T\| \leq 2^{1-1/q} (1 + \vr)$, and
$\|T - S\| \geq \cq/(2^{-1/q} 3^{-(q-1)/q})$
whenever $S$ is disjointness preserving.
In the particular case of $X = L_q$, we can have
$\|T\| \leq 2^{1-1/q}$, and $\|T - S\| \geq 2^{-1/q}$.
\end{proposition}

\begin{remark}\label{r:abram}
Recall that there are no non-zero disjointness preserving
operators from $L_p(0,1)$ to $L_q(0,1)$, when $p < q$ (see \cite{Abr}, and also Proposition \ref{p:norm Lp}.)
\end{remark}

\begin{proof}
Follow the proof of Proposition \ref{p:main}.
Pick $N \in \NN$ so that
$$
\frac{\vr}{2} > \left\{ \begin{array}{ll}
   N^{-1/q}   &   \infty > q \geq 2p,   \\
   \Big( N^{-1} (N+1)^{2 - q/p} \Big)^{1/q}   &   2p > q > p.
\end{array} \right.
$$
Let $M = N(N+1)/2$. Fix $\delta \in (0,1/4)$.
By Krivine's Theorem for lattices (see e.g. \cite{Schep}),
there exist disjoint positive norm one $x_1, \ldots, x_M \in X$
so that, for any $\alpha_1, \ldots, \alpha_M \in \CC$,
$$
\frac{1}{1+\delta} \|\sum_i \alpha_i x_i\| \leq
\big( \sum_i |\alpha_i|^q \big)^{1/q} \leq
(1+\delta) \|\sum_i \alpha_i x_i\| .
$$
Define the operator $T : \ell_p^{N+1} \to X$ by setting
$T \delta_i = N^{-1/q} \sum_{j \in F_i} x_j$,
where $(\delta_i)$ is the canonical basis
for $\ell_p^{N+1}$. Clearly, $T$ is positive.
From the proof of Proposition \ref{p:main},
$\|T\| \leq (1+\delta) 2^{1/q'}$, and $T$ is $(1+\delta)\vr/2-\DP$.

It remains to show that, if $S : \ell_p^{N+1} \to X$ is disjointness
preserving, then $\max_{1 \leq i \leq N} \|(T - S)\delta_i\| \geq \cq/(3 \cdot 2^{-1/q})$.

It is easy to see that any disjoint order bounded sequence in $X$ is
norm null, hence (see e.g. \cite[Section 2.4]{M-N}) $X$ is order continuous.
This, in turn, implies that any ideal in $X$ is a projection band.
For $x \in X$, we shall denote by $\proj_x$ the band projection
corresponding to $x$. Let $P_i = \proj_{S \delta_i} \proj_{T \delta_i}$.
If $P$ is a projection, we use the shorthand $P^\perp = I - P$.
By the basic properties of band projections (see e.g. \cite[Section 1.2]{M-N}),
$P_i$'s are band projections, and $P_i P_j = 0$ if $i \neq j$.

Recall that, for $1 \leq s \leq M$, $Q_s = \{1 \leq i \leq N+1 : s \in F_i\}$,
and $|Q_s| = 2$. Let $y_{is} = P_i x_s$, and note that $y_{is} = 0$
unless $s \in F_i$, or equivalently, $i \in Q_s$. Also let
$y_{0s} = x_s - \sum_{i \in Q_s} P_i x_s = (\sum_{i \in Q_s} P_i)^\perp x_s$.
The elements $y_{is}$ are disjoint. We have
$$  \eqalign{
N^{1/q} \|(T - S) \delta_i\|
&
\geq
N^{1/q} \|\proj_{S \delta_i}^\perp \proj_{T \delta_i} (T \delta_i)\|
\cr
&
=
\|\sum_{s \in F_i} (x_s - y_{is})\| \cr
&
=\|\sum_{s \in F_i} (y_{0s} + y_{i^\prime s})\| ,
}  $$
where $i^\prime$ is such that $Q_s = \{i, i^\prime\}$. By the lower $q$-estimate,
$$
N \|(T - S) \delta_i\|^q \geq
\cq^q \sum_{s \in F_i} \big(\|y_{0s}\|^q + \|y_{i^\prime s}\|^q \big) .
$$
Consequently,
$$  \eqalign{
\cq^{-q} N \sum_{i=1}^{N+1} \|(T - S) \delta_i\|^q & \geq
\sum_{i=1}^{N+1} \sum_{s \in F_i} \big(\|y_{0s}\|^q + \|y_{i^\prime s}\|^q \big)
\cr
&
=
\sum_{s=1}^M \sum_{i \in Q_s} \big(\|y_{0s}\|^q + \|y_{i^\prime s}\|^q \big) \cr
& =
\sum_{s=1}^M \big(2 \|y_{0s}\|^q + \sum_{i \in Q_s} \|y_{i s})\|^q \big) .
}   $$
An easy computation shows that the inequality
$$
2 a^q + b^q + c^q \geq a^q + b^q + c^q \geq 3^{1-q} (a+b+c)^q
$$
holds for any non-negative reals $a, b, c$, hence
$$  \eqalign{
2 \|y_{0s}\|^q + \sum_{i \in Q_s} \|y_{i s}\|^q
&
\geq
3^{1-q} \big( \|y_{0s}\| + \sum_{i \in Q_s} \|y_{is}\| \big)^q
\cr
&
\geq
3^{1-q} \|y_{0s} + \sum_{i \in Q_s} y_{is}\|^q
\cr
&= 3^{1-q} \|x_s\|^q .
}  $$
Therefore,
$$
\cq^{-q} N \sum_{i=1}^{N+1} \|(T - S) \delta_i\|^q \geq
\frac{1}{3^{q-1}} \sum_{s=1}^M \|x_s\|^q = \frac{M}{3^{q-1}} .
$$
Thus, for some $i$,
$$
\|(T - S) \delta_i\|^q \geq \frac{\cq^q M}{3^{q-1} N (N+1)} =
\frac{\cq^q}{2 \cdot 3^{q-1}} .
$$

The particular case of $X = L_q(\mu)$ is more straightforward.
In this case, $\cq = 1$, and the $x_i$'s satisfy $\|\sum_i \alpha_i x_i\| =
\big( \sum_i |\alpha_i|^q \big)^{1/q}$ (that is, we can take $\delta = 0$).
Keeping the same notation as before, we obtain:
$$  \eqalign{
N \sum_{i=1}^{N+1} \|(T - S) \delta_i\|^q
&
\geq
\sum_{s=1}^M \sum_{i \in \{0\} \cup Q_s} \|y_{is}\|^q
\cr
&
=
\sum_{s=1}^M \|\sum_{i \in \{0\} \cup Q_s} y_{is}\|^q \cr
&=
\sum_{i=1}^M \|x_s\|^q = M ,
}  $$
hence, for some $i$, $\|(T - S) \delta_i\|^q \geq M/(N(N+1)) = 1/2$.
\end{proof}

\section{Modulus of an $\vr$-DP operator}\label{s:modulus}

By \cite[Section 3.1]{M-N}, the modulus of a disjointness preserving
operator $T$ exists, and for any $x \geq 0$, $|T| x = |Tx|$.
It is easy to see that $\| |T| \| = \|T\|$, and that $|T|$
preserves disjointness. Conversely, if $|T|$ exists, and is
disjointness preserving, then the same is true for $T$. More generally,
if $|T|$ is $\vr$-DP, then $T$ is $\vr$-DP.
Indeed, suppose $|T|$ is $\vr$-DP, and pick disjoint $x$ and $y$:
$$
\||Tx|\wedge|Ty|\| \leq \||T||x|\wedge|T||y|\| \leq \vr \max\{\|x\|,\|y\|\} .
$$

For operators into Dedekind complete $C(K)$ spaces we have a converse:

\begin{proposition}\label{p:modulus}
Consider $T \in B(E,F)$, where $E$ and $F$ are Banach lattices, and
$F$ is an $M$-space.
If $T \in B(E,F)$ is $\vr$-DP, and the modulus $|T|$ exists,
then $|T|$ is $\vr$-DP.
\end{proposition}

\begin{remark}\label{r:modulus}
Suppose, in Proposition \ref{p:modulus}, $F$ is a Dedekind complete $M$-space,
with a strong order unit (equivalently, $F = C(K)$, where $K$ is a Stonian
compact Hausdorff space, see e.g. \cite[Sections 1.a-b]{LT2}).
Then any operator $T \in B(E,F)$ has modulus $|T|$, and
$\| |T| \| = \|T\|$, see e.g. \cite{WiSUR}.
\end{remark}

\begin{proof}
Recall that for any $x\in E$ we have $|T||x|=\vee_{|y|\leq|x|}|Ty|$.
Now, given disjoint $x_1, x_2$ we have
\begin{align*}
\big\| \big| |T| x_1 \big| \wedge \big| |T| x_2 \big| \big\|
 &\leq \big\| |T| |x_1| \wedge  |T| |x_2| \big\| \\
&=\big\|\vee_{|y_1|\leq|x_1|}|Ty_1|\wedge\vee_{|y_2|\leq|x_2|}|Ty_2|\big\| \\
&=\big\|\vee_{|y_1|\leq|x_1|,|y_2|\leq|x_2|}|Ty_1|\wedge|Ty_2|\big\| .
\end{align*}
As $F$ is an $M$-space,
$$
\big\|\vee_{|y_1|\leq|x_1|,|y_2|\leq|x_2|}|Ty_1|\wedge|Ty_2|\big\| =
\sup_{|y_1|\leq|x_1|,|y_2|\leq|x_2|}\big\||Ty_1|\wedge|Ty_2|\big\| .
$$
Recall that $T$ is $\vr$-DP, hence
$$
\big\||Ty_1|\wedge|Ty_2|\big\| \leq \vr \max\{\|y_1\|,\|y_2\|\}
\leq\vr \max\{\|x_1\|,\|x_2\|\} ,
$$
and therefore, $\big\| \big| |T| x_1 \big| \wedge \big| |T| x_2 \big| \big\| \leq
 \vr \max\{\|x_1\|,\|x_2\|\}$.
\end{proof}

Incidentally, in the non-locally convex setting, we have some stability
for the modulus of an $\vr$-DP operator.

\begin{proposition}
Let $0<p\leq1/2$, a Banach lattice $E$ and $T:\ell_p\rightarrow E$ an $\vr$-DP operator.
The modulus $|T|$ (which is also bounded) is $\sqrt{\vr \|T\|}$-DP.\end{proposition}

\begin{proof}
Let $f_n=T\delta_n$, where $(\delta_n)_{n=1}^\infty$ form the canonical basis
of $\ell_p$. We have that $|T|\delta_n=|f_n|$. Indeed, since $\delta_n$ is an atom we have
$$
|T|\delta_n=\sup\{|Ty|:|y|\leq \delta_n\}=\sup\{|T\lambda\delta_n|:|\lambda|\leq1\}=|T\delta_n|.
$$
Therefore, $|T|:\ell_p\rightarrow E$ is given by $|T|(\sum_n a_n\delta_n)=\sum_n a_n |f_n|$
(which defines a bounded operator).
We claim that, for $n \neq m$,
\begin{equation}
\||a_nf_n|\wedge|b_mf_m|\|\leq \sqrt{\vr \|T\|} \sqrt{|a_n| |b_m|} .
\label{eq:p<1/2}
\end{equation}
Indeed, as $T$ is $\vr$-DP, we have
$\||a_nf_n|\wedge|b_mf_m|\|\leq\vr(|a_n| \vee |b_m|)$.
Also, $\||a_nf_n|\wedge|b_mf_m|\|\leq\|a_nf_n\|\wedge\|b_mf_m\|
 \leq \|T\| |a_n| \wedge |b_m|$.
Assume without loss of generality that $|a_n| \leq |b_m|$. Then
$\||a_nf_n|\wedge|b_mf_m|\|\leq\vr |b_m|\wedge \|T\| |a_n| \leq
\sqrt{\vr |b_m|\|T\| |a_n|}$, establishing \eqref{eq:p<1/2}.

Now, let $x,y\in\ell_p$ be disjoint elements. We can write
$x=\sum_{i\in A} a_i \delta_i$, $y=\sum_{j\in B} b_j \delta_j$ with $A\cap B=\emptyset$.
Taking \eqref{eq:p<1/2} into account, we obtain
\begin{align*}
\Big\|\big||T|x\big|\wedge\big||T|y\big|\Big\|
&\leq\Big\||T|(\sum_{i\in A}|a_i|\delta_i)\wedge|T|(\sum_{j\in B}|b_j|\delta_j)\Big\|  \\
&\leq\Big\|\sum_{i\in A}\sum_{j\in B}|a_if_i|\wedge|b_jf_j|\Big\|\leq
\sum_{i\in A}\sum_{j\in B}\big\||a_if_i|\wedge|b_jf_j|\big\|   \\
&\leq
\sum_{i\in A}\sum_{j\in B} \sqrt{\vr \|T\|} \sqrt{|a_n| |b_m|}
\leq \sqrt{\vr \|T\|} \sqrt{\|x\|_2\|y\|_2}  \\
& \leq
\sqrt{\vr \|T\|} \sqrt{\|x\|_p\|y\|_p} \leq \sqrt{\vr \|T\|} \max\{\|x\|_p,\|y\|_p\}.
\end{align*}
\end{proof}

The result below shows that, in general, the $\vr$-disjointness preserving properties of $T$
do not allow us to conclude anything about the $\vr$-disjointness properties of $|T|$,
even if the latter exists.

\begin{proposition}\label{p:bad mod}
For every $\vr > 0$, there exists an operator $T \in B(\ell_2)$,
so that $\|T\| \geq 1$, $\| |T| \| \leq 2$, $T$ is $\vr$-DP,
yet $|T|$ is not $c$-DP whenever $c \leq 1/2$.
Moreover, $\|T - I_{\ell_2}\| < \vr$, while
$\| |T| - U \| \geq 1/(3\sqrt2)$ whenever $U$ is disjointness preserving.
\end{proposition}

Start by observing that the property of being $\vr$-DP is preserved by direct sums.

\begin{lemma}\label{l:disrum}
Suppose $(E_i)_{i\in\mathbb N}$, $(F_i)_{i\in\mathbb N}$ are Banach lattices, $U$
is a Banach space with a $1$-unconditional basis, and
the operators $T_i \in B(E_i,F_i)$ are such that
$\sup_i \|T_i\| < \infty$. Define the Banach lattices
$E = (\oplus_i E_i)_U$ and $F = (\oplus_i F_i)_U$,
and the operator $T = \oplus_i T_i \in B(E,F)$.
If $T_i$ is $\vr$-DP for every $i\in\mathbb N$, then $T$ is $2\vr$-DP.
\end{lemma}

\begin{proof}
Consider disjoint $x = (x_i)_{i\in\mathbb N}, y = (y_i)_{i\in\mathbb N} \in E$
(here $x_i, y_i \in E_i$, for every $i\in\mathbb N$). By \cite[Proposition 1.d.2]{LT2}, we have
\begin{align*}
\left\| |Tx| \wedge |Ty| \right\| &=\left\| \big(\| |T_i x_i| \wedge |T_i y_i| \|\big)_i \right\|_U\\
&\leq\vr \left\| \Big(\max\{\|x_i\|, \|y_i\|\} \Big)_i \right\|_U \\
& \leq\vr \left\| \big( \|x_i\|+\|y_i\| \big)_i \right\|_U \\
&\leq2\vr \max\{\|x\|, \|y\|\}.
\end{align*}
\end{proof}

\begin{proof}[Proof of Proposition \ref{p:bad mod}]
Consider the operators $S_i \in B(\ell_2^{2^i})$, given by unitary
Walsh matrices. It is known that
$$
|S_i| = 2^{i/2} \xi_i \otimes \xi_i,
$$
where $\xi_i$ is the unit vector $2^{-i/2} \sum_{j=1}^{2^i} e_j$
($e_1, \ldots, e_{2^i}$ is the canonical basis of $\ell_2^{2^i}$).
Let
$$
T_i = I_{\ell_2^{2^i}} + 2^{-i/2} S_i.
$$
Pick $k \in \NN$ so that $2^{-k/6} < \vr/6$.
Identify $(\oplus_{i \geq k} \ell_2^{2^i})_2$ with $\ell_2$,
then we can view $T = \oplus_{i \geq k} T_i$ as an operator on $\ell_2$.
We show that $T$ has the required properties.

Indeed, for any $i$, $\|T_i\| \geq 1 - 2^{-i/2}$, hence
$\|T\| = \sup_i \|T_i\| \geq 1$. Furthermore,
$\|T - I_{\ell^2}\| = \sup_i 2^{-i/2}\|S_i\| \leq \vr$.
The operator $|T| = \oplus_i (I_{\ell_2^{2^i}} + \xi_i \otimes \xi_i)$
has norm $2$.

Now fix $i > k$, and consider disjoint vectors
$x = 2^{-(i-1)/2} \sum_{j=1}^{2^{i-1}} e_j$ and
$y = 2^{-(i-1)/2} \sum_{j=2^{i-1}+1}^{2^i} e_j$
in the unit ball of $\ell_2^{2^i}$.
Then $|T| x = |T| y = 2^{-1/2} \xi_i$, hence
$$
\|||T| x| \wedge ||T| y|\| = 2^{-1/2}.
$$
Thus, $|T|$ cannot be $c$-DP for $c < 1/2$.

To prove that $T$ is $\vr$-DP, it suffices to prove
(in light of Lemma \ref{l:disrum}) that,
for any $i > k$, $I + 2^{-i/2} S_i$ is $\vr/2$-DP. If $x, y \in \ball(\ell_2^{2^i})$
are disjoint, then
$$   \eqalign{
&
\Big| \big( I + 2^{-i/2} S_i \big) x \Big| \wedge
\Big| \big( I + 2^{-i/2} S_i \big) y \Big| \leq
\big( |x| + 2^{-i/2} |S_i x| \big) \wedge \big( |y| + 2^{-i/2} |S_i y| \big)
\cr
&
\leq
|x| \wedge 2^{-i/2} |S_i y|  + 2^{-i/2} |S_i x| \wedge |y| +
2^{-i/2} |S_i x| \wedge 2^{-i} |S_i y| ,
}  $$
hence
$$  \eqalign{
&
\left\| \Big| \big( I + 2^{-i/2} S_i \big) x \Big| \wedge
\Big| \big( I + 2^{-i/2} S_i \big) y \Big| \right\|
\cr
&
\leq
\min\big\{ 2^{-i/2} \|x\|, \|y\| \big\} +
\min\big\{ 2^{-i/2} \|y\|, \|x\| \big\} +
\min\big\{ 2^{-i/2} \|x\|, 2^{-i/2} \|y\|\big\}
\cr
&
\leq
3 \cdot 2^{-i/2} \leq
\vr/2 ,
}  $$
by our choice of $k$.

Finally, suppose $U \in B(\ell_2)$ is a disjointness preserving operator.
Let $V = |T| - U$, and suppose, for the sake of contradiction, that $\|V\| < 1/(3\sqrt{2})$.
As before, take
$x = 2^{-(i-1)/2} \sum_{j=1}^{2^{i-1}} e_j$ and
$y = 2^{-(i-1)/2} \sum_{j=2^{i-1}+1}^{2^i} e_j$. Then
$\|||T| x| \wedge ||T| y|\| = 2^{-1/2}$. On the other hand,
\begin{align*}
(|T| x) \wedge (|T| y)
&
=
(Ux + Vx) \wedge (Uy + Vy) \leq (|Ux| + |Vx|) \wedge (|Uy| + |Vy|)
\\
&
\leq
|Ux| \wedge |Vy| + |Vx| \wedge |Uy| + |Vx| \wedge |Vy| ,
\end{align*}
hence
$$
\frac{1}{\sqrt2} = \|||T| x| \wedge ||T| y|\| \leq \|Vy\| +2\|Vx\| \leq
3\|V\| < \frac{1}{\sqrt2} ,
$$
yielding a contradiction.
\end{proof}

\section{Lattice homomorphisms and operators preserving $p$-estimates}\label{s:lattice p estimates}

Let us consider now positive operators being ``almost lattice homomorphisms.'' We say that an operator $T \in B(E,F)$ is an \emph{$\vr$-lattice homomorphism}
(\emph{$\vr$-LH} for short) if, for any $x \in E$,
$$
\big\|\big|T|x|\big| - |Tx|\big\| \leq \vr \|x\|.
$$
A positive operator $T \in B(E,F)$ is said to be \emph{$\vr$-minimum preserving}
(\emph{$\vr$-MP}) if, for any
positive $x, y \in \ball(E)$,
$$
\| (Tx) \wedge (Ty) - T (x \wedge y)\| \leq \vr.
$$
It is known (see \cite[Section 3.1]{M-N}) that a positive operator
is disjointness preserving if and only if it is $0$-LH,
if and only if it is $0$-MP; in this case, it is a lattice homomorphism.
In the ``approximate'' case, the notions introduced above are connected
to being $\vr^\prime$-DP as well (for some $\vr^\prime$ depending on $\vr$).

\begin{proposition}\label{p:connections}
For Banach lattices $E$ and $F$, and $T \in B(E,F)$, the following holds:
\begin{enumerate}
\item
If $T$ is positive, then $T$ is $\vr$-MP if and only if it is $\vr$-DP.
\item
Any $\vr$-DP operator between real Banach lattices is a $2\vr$-LH.
\item
If $T$ is $\vr$-LH, then $T$ is $4\vr$-DP in the real case,
or $16\vr$-DP in the complex case.
If, in addition, $T$ is positive, then it is $\vr$-DP.
\end{enumerate}
\end{proposition}

\begin{proof}
(1) If $T$ is $\vr$-MP, then it is $\vr$-DP,
by Proposition \ref{p:positive}. To prove the converse, consider
$x, y \in \ball(E)_+$. Then $x_0 = x - x \wedge y$
and $y_0 = x - x \wedge y$ are positive and disjoint, and
$$
Tx \wedge Ty - T(x \wedge y) =
\big(Tx_0 + T(x \wedge y)\big) \wedge \big(Ty_0 + T(x \wedge y)\big) - T(x \wedge y) =
Tx_0 \wedge Ty_0 .
$$
If $T$ is $\vr$-DP, then
$\|Tx \wedge Ty - T(x \wedge y)\| = \|Tx_0 \wedge Ty_0\| \leq \vr$. \\

(2)
Suppose $T$ is a $\vr$-DP map between real Banach lattices.
Then, for any $x \in E$,
\begin{align*}
\big||Tx| - |T|x||\big|
&
=
\big||Tx_+-Tx_-| - |Tx_++Tx_-|\big| =2\big(|Tx_+|\wedge |Tx_-|\big).
\end{align*}
As $\max\{\|x_+\|, \|x_-\|\} \leq \|x\|$, and $x_+\perp x_-$ we have
$\| |Tx| - |T|x|| \| \leq 2 \vr \|x\|$.\\

(3) Suppose $T$ is $\vr$-LH, and pick disjoint positive $y, z \in \ball(E)$.
Let $x = y-z$. As in part (2), we obtain
$$
\| |Ty| \wedge |Tz| \| = \frac12 \||Tx| - |T|x||\| \leq \frac{\vr}{2} \|x\|
 \leq \frac{\vr}{2} (\|y\| + \|z\|) \leq \vr .
$$
To finish the proof, apply Proposition \ref{p:positive}.
\end{proof}

In the rest of the section we consider operators which almost preserve estimates
of the form $(|x|^p+|y|^p)^{1/p}$, and their connection with $\vr$-DP operators and lattice homomorphisms. This approach is in part motivated by Corollary \ref{c:arb number}. In particular, this will allow us to extend some of the previous results to the complex setting (see Proposition \ref{complexDP}.)

Given $1\leq p\leq \infty$, a positive operator between Banach lattices $T:E\rightarrow F$
is said to be \emph{$\vr$ preserving $p$-estimates} if for every $x,y\in E$ we have
$$
\Big\|T\Big(|x|^p+|y|^p\Big)^{\frac1p}-\Big(|Tx|^p+|Ty|^p\Big)^{\frac1p}\Big\|\leq \vr(\|x\|+\|y\|),
$$
while for $p=\infty$, we would have
$$
\Big\|T\Big(|x|\vee|y|\Big)-\Big(|Tx|\vee|Ty|\Big)\Big\|\leq \vr(\|x\|+\|y\|).
$$

It is easy to see that an operator is $\vr$ preserving $1$-estimates if and only if it is an $\vr$-lattice homomorphism. More generally, we have

\begin{proposition}\label{prop:p-estimates}
Let $E$ and $F$ be real Banach lattices. If $T \in B(E,F)$ is a positive
$\vr$-DP operator, then for every $1<p<\infty$, $T$ is
$K \log_2 (\vr (\|T\|+1))^{-1} (\vr (\|T\|+1))^{1/2}$ preserving $p$-estimates
(where $K$ is a universal constant).
\end{proposition}

Recall that according to Proposition \ref{p:connections}(1), a positive operator is $\vr$-MP if and only if it is $\vr$-DP. Before giving the proof, we need a preliminary Lemma:

\begin{lemma}\label{l:sup n}
If $T \in B(E,F)$ is a positive $\vr$-MP operator, then, for any
$x_1, \ldots, x_n \in \ball(E_+)$, we have
$$
\big\| T (\vee_{i=1}^n x_i) - \vee_{i=1}^n T x_i \big\| \leq \vr \lceil \log_2 n \rceil n .
$$
\end{lemma}

\begin{proof}
It suffices to show that, for any $m \in \NN$,
\begin{equation}
\label{eq:vee 2m}
\big\| T (\vee_{i=1}^{2^m} x_i) - \vee_{i=1}^{2^m} T x_i \big\|
 \leq \vr m 2^{m-1} .
\end{equation}
Proceed by induction on $m$.
The case of $m = 1$ is contained in the definition of $T$
being $\vr$-MP. To deal with the induction step, suppose
the statement holds for $m$, and prove it for $m+1$. For $j = 0,1$ let
$$
y_j = \vee_{i = 2^m j + 1}^{2^m j + 2^m} x_i  {\textrm{  and   }}
z_j = T y_j - \vee_{i = 2^m j + 1}^{2^m j + 2^m} T x_i .
$$
By the induction hypothesis, $\|z_j\| \leq \vr m 2^{m-1}$
(and it is easy to see that $z_j \geq 0$). Also,
$$
\|T (y_0 \vee y_1) - (T y_0) \vee (T y_1)\| \leq \vr \max\{\|y_0\|, \|y_1\|\} \leq 2^m \vr .
$$
We clearly have
\begin{align*}
T (\vee_{i=1}^{2^m} x_i) - \vee_{i=1}^{2^m} T x_i
&
=
T(y_0 \vee y_1) - (T y_0 - z_0) \vee (T y_1 - z_1)
\\
&
\leq
\big( T (y_0 \vee y_1) - (T y_0) \vee (T y_1) \big) + z_0 + z_1 ,
\end{align*}
hence,
$$
\big\| T (\vee_{i=1}^{2^m} x_i) - \vee_{i=1}^{2^m} T x_i \big\| \leq
\big\| T (y_0 \vee y_1) - (T y_0) \vee (T y_1) \big\| + (\|z_0\| + \|z_1\|).
$$
From the above,
$$
\big\| \vee_{i=1}^n T x_i - T (\vee_{i=1}^n x_i) \big\| \leq
2^m \vr + 2 \cdot m 2^{m-1} \vr = (m+1) 2^m \vr .
$$
\end{proof}

We also need a simple calculus result.

\begin{lemma}\label{l:calculus}
Suppose $\phi$ is a monotone continuous function on an interval $[a,b]$,
continuously differentiable on $(a,b)$. Then the arclength of the graph
of $\phi$ does not exceed $b-a + |\phi(b) - \phi(a)|$.
\end{lemma}

\begin{proof}
For the arclength in question we have
$$
L = \int_a^b \sqrt{1 + (\phi^\prime(t))^2} \, dt \leq
\int_a^b \big(1 + |\phi^\prime(t)|\big) \, dt .
$$
The monotonicty of $\phi$ implies
$\int_a^b |\phi^\prime(t)| \, dt = |\phi(b) - \phi(a)|$.
\end{proof}

\begin{proof}[Proof of Proposition \ref{prop:p-estimates}]
For any $u$ and $v$ in a Banach lattice, if $1/p+1/q=1$, then (see \cite[1.d]{LT2})
$$
\big(|u|^p + |v|^p \big)^{1/p} = \bigvee \big\{ \alpha |u| + \beta |v| :
 \alpha , \beta \in [0,1] , \alpha^q + \beta^q = 1 \big\} .
$$

For any $N \in \NN$, let $\{(x_j,y_j):j=0,1,\ldots N\}$
be a collection of points satisfying $x_j,y_j\in[0,1]$, $x_j^q+y_j^q=1$
and such that for any $(\alpha,\beta)$ with $\alpha,\beta\in[0,1]$ and
$\alpha^q+\beta^q=1$, there exists $0\leq j\leq N$ for which
$$
\max\{|\alpha-x_{j}|,|\beta-y_{j}|\}\leq \frac{C_q}{N},
$$
where $C_q$ is the length of the curve $\{(x,y):x,y\in[0,1],\,x^q+y^q=1\}$.
By Lemma \ref{l:calculus}. $C_q \leq 2$.
Thus, for any $(\alpha,\beta)$ with $\alpha,\beta\in[0,1]$ and
$\alpha^q+\beta^q=1$ there exists $j$ so that
$$
\alpha |u| + \beta |v| \leq
 \big( x_j |u| + y_j|v| \big) +
 \frac{2}{N} \big( |u| + |v| \big) .
$$
Taking the supremum, we obtain
$$\bigvee \big\{ \alpha |u| + \beta |v| :
 \alpha , \beta \in [0,1] , \alpha^q + \beta^q = 1 \big\}
\leq
\bigvee_{j=0}^N \big( x_j |u| + y_j |v| \big) +
 \frac{2}{N} \big( |u| + |v| \big) ,$$

and by the triangle inequality we get
\begin{equation}
\label{eq:sqrt est}
\Big\| \big( |u|^p + |v|^p \big)^{1/p} -
 \bigvee_{j=0}^N \big( x_j |u| + y_j |v| \big) \Big\|
  \leq \frac{2}{N} \big(\|u\| + \|v\|\big) .
\end{equation}

By Lemma \ref{l:sup n},
$$
\Big\| T \bigvee_{j=0}^N \big( x_j |x| + y_j |y| \big) -
\bigvee_{j=0}^N \big( x_j T|x| + y_j T|y| \big) \Big\| \leq
\vr 2^{\frac1p} \lceil \log_2 (N+1) \rceil (N+1) .
$$

By Proposition \ref{p:connections} (2), $T$ is $\vr$-LH, hence
$\|T|x| - |Tx|\| , \|T|y| - |Ty|\| \leq \vr$, hence
$$
\Big\| \bigvee_{j=0}^N \big( x_j T|x| + y_j T|y| \big) -
\bigvee_{j=0}^N \big( x_j |Tx| + y_j |Ty| \big) \Big\|
 \leq 2(N+1)\vr .
$$

Thus, by the triangle inequality,
\begin{align*}
&
\Big\| T \big( |x|^p + |y|^p \big)^{1/p} - \big( |Tx|^p + |Ty|^p \big)^{1/p} \Big\|
\leq
\\
&
\Big\| T \big( |x|^p + |y|^p \big)^{1/p} -
T \bigvee_{j=0}^N \big( x_j |x| + y_j |y| \big) \Big\|
+
\\
&
\Big\|
T \bigvee_{j=0}^N \big( x_j |x| + y_j |y| \big) -
 \bigvee_{j=0}^N \big( x_j T|x| + y_j T|y| \big)
\Big\|
+
\\
&
\Big\|  \bigvee_{j=0}^N \big( x_j T|x| + y_j T|y| \big) -
 \bigvee_{j=0}^N \big( x_j |Tx| + y_j |Ty| \big)
 \Big\|
+
\\
&
\Big\|  \bigvee_{j=0}^N \big( x_j |Tx| + y_j |Ty| \big) -
  \big( |Tx|^p + |Ty|^p \big)^{1/p}  \Big\|
\leq
\\
&
\frac{4\|T\|}{N} + \vr(N+1)\Big(2^{\frac1p} \lceil \log_2 (N+1) \rceil + 2\Big).
\end{align*}
To finish the proof, select $N \sim (\vr (\|T\|+1))^{-1/2}$.
\end{proof}

As a consequence of this result, we can now give the complex version of Proposition \ref{p:connections}(2). We follow \cite{AA} in representing a complex Banach lattice $X$ as a complexification of its real part $X_{\RR}$. More precisely, any $x \in X$ can be represented as $x = a + \iota b$, with $a,b \in X_{\RR}$. Then $|x| = (a^2 + b^2)^{1/2}$.

\begin{proposition}\label{complexDP}
Suppose $E$ and $F$ are complex Banach lattices, and $T \in B(E,F)$ is a positive $\vr$-DP operator. Then $T$ is a
$C \log_2 (\vr (\|T\|+1))^{-1} (\vr (\|T\|+1))^{1/2}$-LH (with $C$ a universal constant).
\end{proposition}

\begin{proof}
Consider $T \in B(E,F)$ as in the statement,
and show that, for any $x \in \ball(E)$,
$$
\|T|x| - |Tx|\| \leq C \log_2 (\vr (\|T\|+1))^{-1} (\vr (\|T\|+1))^{1/2} .
$$

By Proposition \ref{p:connections}(1,3), $T|_{E_{\RR}}$ is $2\vr$-LH,
hence by Proposition \ref{prop:p-estimates}, it follows that $T|_{E_{\RR}}$ is
$K \log_2 (2\vr (\|T\|+1))^{-1} (2\vr (\|T\|+1))^{1/2}$ preserving 2-estimates.

Now, write $x = a + \iota b$, where $a$ and $b$ belong to $E_{\RR}$. We have that
\begin{eqnarray*}
\|T|x| - |Tx|\| &= &\Big\| T \big( a^2 + b^2 \big)^{1/2} - \big( (Ta)^2 + (Tb)^2 \big)^{1/2} \Big\|\\
& \leq& 2K_2 \log_2 (2\vr (\|T\|+1))^{-1} (2\vr (\|T\|+1))^{1/2}.
\end{eqnarray*}

\end{proof}

Motivated by Lemma \ref{l:sup n} we will consider next a strengthening of operators $\vr$ preserving $\infty$-estimates. For $\vr > 0$, we say that a positive operator $T \in B(E,F)$ ($E$ and $F$ are
Banach lattices) is $\vr$-\emph{strongly maximum preserving}
($\vr$-\emph{SMP} for short) if, for
any $x_1, \ldots, x_n \in \ball(E)_+$, we have
$$
\|T(\vee_{i=1}^n x_i) - \vee_{i=1}^n T x_i\| \leq \vr.
$$

We say that $T \in B(E,F)$ is a $\vr$-\emph{strongly disjointness preserving}
($\vr$-\emph{SDP}) if, for
any mutually disjoint $x_1, \ldots, x_n \in \ball(E)$, we have
$$
\|\sum_{i=1}^n |T x_i| - \vee_{i=1}^n |T x_i|\| \leq \vr.
$$
Clearly any $\vr$-SMP positive operator is also $\vr$-SDP.

Note that these properties are much harder to satisfy.
For instance, it is easy to see that any operator $T$ is $\|T\|$-DP.
On the other hand, for a pair of Banach lattices $(E,F)$, the following are equivalent:
(1) $E$ is lattice isomorphic to an $M$-space, and (2) There exists $C > 0$
so that any $T \in B(E,F)_+$ is $C \|T\|$-SDP.

To prove $(1) \Rightarrow (2)$, suppose $E$ is an $M$-space.
Fix a positive operator $T : E \to F$, and consider mutually disjoint
$x_1, \ldots, x_n \in \ball(E)$. Then
$$
\big\| \sum_i |T x_i| \big\| \leq \big\| T \sum_i |x_i| \big\| \leq
\|T\| \| \sum_i |x_i| \| = \|T\| \max_i \|x_i\| ,
$$
which implies (2).

For $(2) \Rightarrow (1)$, recall that, by \cite[Sections 2.1, 2.8]{M-N},
the following are equivalent: (i) $E$ is lattice isomorphic to an $M$-space;
(ii) there exists a constant $K$ so that the inequality
$\|\sum_i x_i\| \leq K \max_i \|x_i\|$ holds whenever $x_1, \ldots, x_n \in E$
are mutually disjoint; (iii) there exists a constant $K$ so that the inequality
$\|\sum_i x_i^*\| \geq K^{-1} \sum_i \|x_i^*\|$ holds whenever
$x_1^*, \ldots, x_n^* \in E^*$ are mutually disjoint.
Suppose now that (1) fails, and show that (2) fails as well.

If (1) fails, then for every $C > 1$ there exist mutually disjoint non-zero
$x_1^*, \ldots, x_n^* \in E^*_+$, satisfying
$$
\Big\|\sum_i x_i^*\Big\| < (C+2)^{-1} \sum_i \|x_i^*\|.
$$
Without loss of generality, we can assume $1 = \max_i \|x_i^*\|$.
Applying \cite[Proposition 1.4.13]{M-N} to $x_i^*/\|x_i^*\|$, we see that, for any
$\sigma > 0$, there exist mutually disjoint $x_1, \ldots, x_n \in \ball(E)_+$
so that $\langle x_i^*, x_i \rangle > \|x_i^*\| - n^{-1}$ for any $i$.

Now let $x^* = \sum_i x_i^*$, pick a norm one positive $y \in F$, and define
$T : E \to \spn[y] \subset F : x \mapsto \langle x^*, x \rangle y$.
Clearly $\|T\| = \|x^*\|$. On the other hand, $\max_i \|x_i\| = 1$,
$\vee_i Tx_i \leq y$, and
$$
\sum_i Tx_i = \big(\sum_i \sum_j \langle x_i^*, x_j \rangle \big) y \geq
\Big(\sum_i \|x_i^*\| - 1\Big) y .
$$
Consequently, if $T$ is $\gamma \|T\|$-SDP, then
$$
\gamma \geq \frac{\sum_i \|x_i^*\| - 2}{\|\sum_i x_i^*\|} > C .
$$
As $C$ can be arbitrarily large, we are done.

\begin{theorem}\label{t:SMP}
Suppose $E$ and $F$ are Banach lattices,
and $T \in B(E,F)$ is a positive $\vr$-SDP operator.
\begin{enumerate}
\item
Suppose $E$ is finite dimensional. Then there exists
a disjointness preserving $S \in B(E,F)$ so that $0 \leq S \leq T$,
and $\|T - S\| \leq 2\vr$.
\item
Suppose the order on $E$ is determined by its $1$-unconditional basis,
while $F$ has the Fatou Property
with constant $\fatou$. Then there exists a disjointness preserving
$S \in B(E,F)$ so that $0 \leq S \leq T$, and $\|T - S\| \leq 2 \fatou \vr$.
\end{enumerate}
\end{theorem}

\begin{remark}\label{r:SDP vs DP}
By Corollary \ref{c:arb number}, if a positive operator $T$ is
$\vr$-DP, then for any mutually disjoint
$x_1, \ldots, x_n \in \ball(E)$, we have
$\|\sum_{i=1}^n |T x_i| - \vee_{i=1}^n |T x_i|\| \leq 256 \vr \|\sum_i x_i\|$.
In particular this holds for the operator $T$ from Proposition \ref{p:main}.
However, in light of Theorem \ref{t:SMP}, if $T$ is $\sigma$-SDP, then
$\sigma > 1/4$. Thus, there is no function $f : (0,\infty) \to (0,\infty)$, with
$\lim_{t \to 0} f(t) = 0$, so that being $\vr$-DP implies being $f(\vr)$-SDP.
\end{remark}

\begin{proof}
(1)
It is well known (see e.g. \cite[Corollary 4.20]{Schw})
that $X$ has a basis of atoms, which we denote by $(\delta_i)_{i=1}^n$
($n = \dim X$), and that form a 1-unconditional basis. Use scaling to assume that $T$ is contractive.
Let $f_i = T \delta_i$
As in the proof of Theorem \ref{t:c_0}, define the function $\phi_n : \RR^n \to \RR$ by setting
$$
\phi_n : (t_1, \ldots, t_n) \mapsto
\left\{ \begin{array}{ll}
   0                             &   t_1 \leq \vee_{i=2}^n |t_i|,      \\
   2(t_1 - \vee_{i=2}^n |t_i|)
       &   \vee_{i=2}^n |t_i| \leq t_1 \leq 2 \vee_{i=2}^n |t_i|,      \\
   t_1                           &   t_1 \leq \vee_{i=2}^n |t_i|.
\end{array} \right.
$$
For $1 \leq i \leq n$ set
$$
g_i = \phi_n(f_i, f_{i+1}, \ldots, f_n, f_1, \ldots, f_{i-1}).
$$
We claim that the operator $S : E \to F : \delta_i \mapsto g_i$
has the desired properties.

Note that $0 \leq \phi_n(t_1, \ldots, t_n) \leq t_1$, hence $0 \leq g_i \leq f_i$,
which shows that $0 \leq S \leq T$.

To show that $S$ is disjointness preserving, consider $i \neq j$. Note that, for any
$(t_1, \ldots, t_n) \in \RR^n$,
$$
\phi_n(t_i, t_{i+1}, \ldots, t_n, t_1, \ldots, t_{i-1}) \wedge
\phi_n(t_j, t_{j+1}, \ldots, t_n, t_1, \ldots, t_{j-1}) = 0 ,
$$
hence $g_i$ and $g_j$ are disjoint.

Now note that
$$
\|T-S\| \leq \|(T-S) \sum_{i=1}^n \delta_i\| = \|\sum_{i=1}^n (f_i-g_i)\| .
$$
It therefore suffices to show that
\begin{equation}
\label{eq:fun leq}
\sum_{i=1}^n (f_i-g_i) \leq 2 \big( \sum_{i=1}^n f_i - \vee_{i=1}^n f_i \big) .
\end{equation}
Indeed, applying the definition of $\vr$-SDP to $x_i = \delta_i$, we obtain
$$
\big\| \sum_{i=1}^n f_i  - \vee_{i=1}^n f_i \big\| \leq \vr .
$$
To establish \eqref{eq:fun leq}, by functional calculus it suffices to show
that, for any $t_1, \ldots, t_n \in \RR^n$,
$$  \eqalign{
\big(t_1 - \phi_n(t_1, t_2, \ldots, t_n)\big)
+
&
\big(t_2 - \phi_n(t_2, t_3, \ldots, t_n, t_1)\big) +
\ldots
\cr
+
&
\big(t_n - \phi_n(t_n, t_1, \ldots, t_{n-1})\big) \leq
2 \big( \sum_{i=1}^n t_i - \bigvee_{i=1}^n t_i \big) .
}  $$
By relabeling, we can assume that $t_1 \geq t_2 \geq \ldots \geq t_n$.
In the left hand side, the $i$-th term equals $t_i$, while the first
term doesn't exceed $t_2$. Thus, the left hand side doesn't exceed
$2t_2 + t_3 + \ldots + t_n$
On the other hand, the right hand side equals $2 \sum_{i=2}^n t_i$.

(2) Now denote the basis of $X$ by $(\delta_i)_{i=1}^\infty$,
and set $f_i = T \delta_i$. With the notation of (1), set
$g_i^{(n)} = \phi_n(f_i, f_{i+1}, \ldots, f_n, f_1, \ldots, f_{i-1})$.
Note that, for any $t_1, \ldots, t_{n+1} \in \RR$, we have
$$
\phi_n(t_1, t_2, \ldots, t_n) =
\phi_{n+1}(t_1, t_2, \ldots, t_n, 0) \geq
\phi_{n+1}(t_1, t_2, \ldots, t_n, t_{n+1}) ,
$$
hence we have
$$
f_i \geq g_i^{(i)} \geq g_i^{(i+1)} \geq g_i^{(i+2)} \geq \ldots \geq 0 .
$$
By the $\sigma$-Dedekind completeness of $F$, $g_i = \lim_n g_i^{(n)}$
exists for every $i$. Define $S : E \to F$ by setting $S \delta_i = g_i$.
Clearly $0 \leq S \leq T$. Furthermore, $S$ is disjointness preserving.
Indeed, if $i \neq j$, and $n \geq i \vee j$, then for any
$t_1, \ldots, t_n \in \RR$,
$$
\phi_n(t_i, t_{i+1}, \ldots, t_n, t_1, \ldots, t_{i-1}) \wedge
\phi_n(t_j, t_{j+1}, \ldots, t_n, t_1, \ldots, t_{j-1}) = 0 ,
$$
hence $g_i^{(n)} \wedge g_j^{(n)} = 0$.

To estimate $\|T-S\|$, note that
$$
\|T-S\| \leq \sup_m \big\| (T-S) \sum_{i=1}^m \delta_i \big\| =
\sup_m \big\| \sum_{i=1}^m (f_i - g_i) \big\| .
$$
For each $m$,
$$
\big\| \sum_{i=1}^m (f_i - g_i) \big\| \leq
\fatou \sup_n \big\| \sum_{i=1}^m (f_i - g_i^{(n)}) \big\| .
$$
By the proof of part (1),
$$
\big\| \sum_{i=1}^m (f_i - g_i^{(n)}) \big\| \leq 2 \vr ,
$$
and the proof is complete.
\end{proof}

\end{document}